%% file: ggv2011-d.tex
\documentclass[1p,authoryear,letterpaper]{elsarticle}

\input{ggv2010-setup}

\input{ggv2010-custom}

\usepackage{url}

\newcommand{\mLambdahat}{\mat{\hat{\Lambda}}}
\newcommand{\mXhat}{\mat{\hat{X}}}
\newcommand{\mRhat}{\mat{\hat{R}}}
\newcommand{\mYhat}{\mat{\hat{Y}}}
\newcommand{\mZhat}{\mat{\hat{Z}}}
\newcommand{\mQhat}{\mat{\hat{Q}}}
\newcommand{\mHhat}{\mat{\hat{H}}}
\newcommand{\mBhat}{\mat{\hat{B}}}
\newcommand{\vyhat}{\vhat{y}}
\newcommand{\vahat}{\vhat{a}}
\providecommand{\cmLambda}{\ensuremath{\celm{\Lambda}}}
\newcommand{\Matlab}{\textsc{Matlab}\xspace}

\DeclareMathSymbol{\minus}{\mathord}{operators}{"2D}
\DeclareMathOperator{\tangle}{\texttt{angle}}

\usepackage{geometry}
\usepackage{algorithmic}
\usepackage{epstopdf}

\begin{document}

\begin{frontmatter}

\title{The power and Arnoldi methods in an algebra of circulants}
\author{David F.~Gleich\footnotemark[1]\footnotemark[2]}
\address{Sandia National Labs\footnotemark[3], Livermore, CA, United States}
\ead{dfgleic@sandia.gov}

\author{Chen Greif\footnotemark[2]}
\ead{greif@cs.ubc.ca}

\author{James M.~Varah\footnotemark[2]}
\address{The University of British Columbia, Vancouver, BC, Canada}
\ead{varah@cs.ubc.ca}

\begin{abstract}
Circulant matrices play a central role in a recently proposed
formulation of three-way data computations.
In this
setting, a three-way table corresponds to a matrix where each
``scalar'' is a vector of parameters defining a circulant.
This interpretation
provides many generalizations of results from matrix or
vector-space algebra. We derive the power and Arnoldi methods
in this algebra.
In the course of our derivation, we define inner products, norms, and other notions.
These extensions are straightforward in an
algebraic sense, but the implications are dramatically different
from the standard matrix case. For example, a matrix of
circulants has a polynomial number of eigenvalues in its
dimension; although, these can all be represented by a
carefully chosen canonical set of eigenvalues and vectors.
These results and algorithms are closely related to standard decoupling
techniques on block-circulant matrices using the fast Fourier transform.
\end{abstract}

\begin{keyword}
 block-circulant \sep circulant module \sep tensor
 \sep FIR matrix algebra \sep power method \sep Arnoldi process
\end{keyword}

\renewcommand{\thefootnote}{\fnsymbol{footnote}}
\footnotetext[1]{\scriptsize Corresponding author.  
Half of this author's work was conducted at the University
of British Columbia.}
\footnotetext[2]{\scriptsize The work of this author was
 supported in part by the Natural Sciences and Engineering Research
 Council of Canada}
\footnotetext[3]{\scriptsize Sandia National Laboratories is a multi-program laboratory 
managed and operated by Sandia Corporation, a wholly owned subsidiary of 
Lockheed Martin Corporation, for the U.S. Department of Energy's National 
Nuclear Security Administration under contract DE-AC04-94AL85000.}

\end{frontmatter}

\section{Introduction}

We study iterative algorithms in a circulant algebra, which is
a recent proposal for a set of operations that generalize matrix
algebra to three-way data \cite{kilmer2008-circ-tensor-svd}.
In particular, we extend this algebra
with the ingredients required for iterative methods such as the
power method and Arnoldi method, and study the behavior of
these two algorithms.

Given an $m \times n \times k$ table of data, we view this data
as an $m \times n$ matrix where each ``scalar'' is a vector of
length $k$.  We denote the space of length-$k$ scalars as $\KK_k$.
These scalars interact like circulant matrices.
Circulant matrices are a commutative, closed class under the standard matrix operations.
Indeed, $\KK_k$ is the ring of circulant matrices, where we
identify each circulant
matrix with the $k$ parameters defining it.

Formally, let $\calpha \in \KK_k$.
Elements in the circulant algebra are denoted by an underline to distinguish
them from regular scalars.  When an element is written with an explicit
parameter set, it is denoted by braces, for example $$\calpha =
\csbmat{ \alpha_1 & \ldots & \alpha_k }.$$ In what follows, we will use the notation $\circeq$ to provide an equivalent matrix-based notation for an operation involving $\KK_k$.
We define the operation $\tcirc(\cdot)$
as the ``circulant matrix representation'' of a scalar:
\begin{equation} \label{eq:circ-op-start}
 \calpha \quad \circeq \quad \tcirc(\calpha)  \eqdef
\bmat{ \alpha_1 & \alpha_{k} & \ldots & \alpha_2 \\
        \alpha_2 & \alpha_1 & \ddots & \vdots \\
        \vdots & \ddots & \ddots & \alpha_k \\
        \alpha_k & \ldots &\alpha_2 & \alpha_1} .
\end{equation}				
Let $\calpha$ be as above, and also let  $\cbeta \in \KK_k$.
The basic addition  and multiplication operations between scalars are then
\begin{equation}
 \calpha + \cbeta \circeq \tcirc(\calpha) + \tcirc(\cbeta)
\quad \text{ and } \quad \calpha \circ \cbeta \circeq \tcirc(\calpha) \tcirc(\cbeta).
\end{equation}
We use here a special symbol, the $\circ$ operation, to denote the product between these scalars, highlighting the difference from the standard matrix product.
Note that the element
\[ \cel{1} = \csbmat{ 1 & 0 & \ldots & 0 } \] is
the multiplicative identity.

Operations between vectors and matrices have similar, matricized, expressions. 
We use $\KK_k^{n}$ to denote the space of length-$n$ vectors where each
component is a $k$-vector in $\KK_k$, and  $\KK_k^{m \times n}$ to denote the space of $m \times n$ 
matrices of these $k$-vectors.  Thus, we identify each $m \times n \times k$ table 
with an element of $\KK_k^{m \times n}$.  Let $\cmA \in \KK_k^{m \times n}$ and $\cvx \in \KK_k^{n}$.  Their product is:
\begin{equation}
 \cmA \circ \cvx =
  \sbmat{ \sum_{j=1}^n \cel{A}_{1,j} \;\circ\; \cel{x}_j \\
                \vdots \\
                \sum_{j=1}^n \cel{A}_{m,j} \;\circ\; \cel{x}_j }
  \circeq
  \sbmat{ \tcirc( \cel{A}_{1,1} ) & \ldots & \tcirc( \cel{A}_{1,n} ) \\
                 \vdots & \ddots & \vdots \\
                 \tcirc( \cel{A}_{m,1} ) & \ldots & \tcirc( \cel{A}_{m,n} ) }
  \sbmat{ \tcirc( \cel{x}_{1} ) \\
                \vdots \\
                \tcirc( \cel{x}_{n} ) } .
\end{equation}
Thus, we extend the operation $\tcirc$ to matrices and vectors of $\KK_k$ scalars so that
\begin{equation}
 \tcirc(\cmA) \eqdef
\sbmat{ \tcirc( \cel{A}_{1,1} ) & \ldots & \tcirc( \cel{A}_{1,n} ) \\
                 \vdots & \ddots & \vdots \\
                 \tcirc( \cel{A}_{m,1} ) & \ldots & \tcirc( \cel{A}_{m,n} ) }
\quad \text{ and } \quad
\tcirc(\cvx) \eqdef
  \sbmat{ \tcirc( \cel{x}_{1} ) \\
                \vdots \\
                \tcirc( \cel{x}_{n} ) }.
\end{equation}
The definition of the product can now be compactly written as
\begin{equation} \label{eq:circ-op-end}
 \cmA \circ \cvx \circeq \tcirc(\cmA) \tcirc(\cvx).
\end{equation}
Of course this notation also holds for the special case of scalar-vector multiplication.  Let $\calpha \in \KK_k$.  Then
\[ \cvx \circ \calpha \circeq \tcirc(\cvx) \tcirc(\calpha). \]

The above operations define the basic computational routines to treat 
$m \times n \times k$  arrays as $m \times n$ matrices of $\KK_k$.  
They are equivalent to those proposed by \citet{kilmer2008-circ-tensor-svd}, 
and they constitute a module over vectors composed of circulants, as 
shown recently in~\citet{braman201x-tensor-eigenvalues}.  Based on this 
analysis, we term the set of operations the \emph{circulant algebra}.  
We note that these operations have more efficient implementations,
which will be discussed in Sections~\ref{sec:fft}~and~\ref{sec:computations}.

The circulant algebra analyzed in this paper is closely related to the
\emph{FIR matrix algebra} due to
\citet[Chapter 3]{lambert1996-thesis}.
Lambert proposes an algebra
of circulants; but his circulants are
padded with additional zeros to better
approximation a finite impulse response
operator.  He uses it to study
blind deconvolution problems~\cite{lambert2001-polynomials-svd}.
As he observed, the relationship with matrices implies that many
standard decompositions and techniques from real or complex
valued matrix algebra carry over to the circulant algebra.

The circulant algebra in this
manuscript is a particular instance of a matrix-over-a-ring,
a long studied generalization of linear algebra
\cite{McDonald1984-ring-algebra, brewer1986-linear-systems}.
Prior work focuses on
Roth theorems for the
equation $AX - XB=C$ \cite{Gustafson1979-Roth-theorems}; generalized inverses
\cite{Prasad1994-generalized}; completion and
controllability problems \cite{Gurvits1992-controllability};
matrices over the ring of integers for
computer algebra systems \cite{Hafner1991-matrices-over-rings};
and transfer functions and linear dynamic systems \cite{Sontag1976-ring-systems}.
Finally, see \citet{Gustafson1991-modules-and-matrices} for some interesting relationships between vectors space
theory and module theory.
A recent proposal extends many of
the operations in \citet{kilmer2008-circ-tensor-svd}
to more general algebraic structures \cite{Navasca2010-modules}.

Let us provide some further context on related work.
Multi-way arrays, tensors, and hypermatrices
are a burgeoning area of research;
see \citet{kolda2009-tensor-decompositions}
for a recent  comprehensive survey.  Some
of the major themes are multi-linear
operations, fitting multi-linear
models, and multi-linear generalizations
of eigenvalues \cite{Qi2007-tensor-eigenvalues}.
The formulation in this paper gives rise to stronger
relationships with the literature
on block-circulant matrices,
which have
been studied for quite some time.
See \citet{tee2005-block-circulant} and the references therein 
for further historical and mathematical context on circulant matrices.
In particular, \citet{baker1989-block-circulant-svd} gives a procedure
for the SVD of a block circulant that involves using
the fast Fourier transform to decouple
the problem into independent sub-problems, just
as we shall do throughout this manuscript.
Other work in this vein includes solving
block-circulant systems that arise in
the theory of antenna arrays: \cite{sinott1973-antenna-arrays,%
mazancourt1983-block-circulant,vescovo1997-block-circulant}.

The remainder of this paper is structured as follows.
We first derive a few necessary operations in Section~\ref{sec:ops}, 
including an inner product and norm. 
We then continue this discussion by studying these same operations 
using the Fourier transform of the underlying circulant matrices 
(Section~\ref{sec:fft}).  A few theoretical properties of eigenvalues 
in the circulant algebra are analyzed in Section~\ref{sec:eigen}.  
That section is a necessary prelude to the subsequent discussion of 
how the power method~\cite{vonMises1929-power} and the Arnoldi 
method~\cite{Krylov1931-equations,lanczos1950-iteration,Arnoldi1951-minimized}
 generalize to this algebra, 
which comes in Section~\ref{sec:algs}.
We next explain how we implemented these operations in a \Matlab 
package (Section~\ref{sec:computations}); and we provide a numerical 
example of the algorithms (Section~\ref{sec:example}).  
Section~\ref{sec:conclusion} concludes the manuscript with some 
ideas for future work.

\section{Operations with the power method}
\label{sec:ops}

In the introduction, we provided the basic set of operations in the circulant algebra 
(eqs.~\eqref{eq:circ-op-start}-\eqref{eq:circ-op-end}).  We begin this section by stating the standard power method, and then follow by deriving the operations it requires.

Let $\mA \in \RR^{n \times n}$ and let $\vx \in \RR^{n}$ be an arbitrary starting vector.  Then the power method proceeds by repeated applications of $\mA$; see Figure~\ref{fig:power-standard} for a standard algorithmic description. (Line \ref{alg:power:conv} checks for convergence and is one of several possible stopping criteria.) Under mild and well-known conditions (see \citet{Stewart2001-eigensystems}), this iteration converges to the eigenvector with the largest magnitude eigenvalue.

\begin{figure}
\caption{The power method for a matrix $\mA \in \RR^{n \times n}$.
}
\label{fig:power-standard}
\begin{algorithmic}[1]
\REQUIRE  $\mA, \vx\itn{0}, \tau$
\STATE $\vx\itn{0} \leftarrow \vx\itn{0} \normof{\vx\itn{0}}^{-1}$
\FOR {$k=1, \ldots, $ until convergence}
  \STATE $\vy\itn{k} \leftarrow \mA  \vx\itn{k-1}$
  \STATE $\alpha\itn{k} \leftarrow \normof{\vy\itn{k}}$
  \STATE $\vx\itn{k} \leftarrow \vy\itn{k}  {\alpha\itn{k}}^{-1}$
  \IF {${\,\|{\sign(x_1\itn{k}) \vx\itn{k} - \sign(x_1\itn{k-1}) \vx\itn{k-1}}\|\,} < \tau$} \label{alg:power:conv}
    \RETURN $\vx\itn{k}$
  \ENDIF
\ENDFOR
\end{algorithmic}
\end{figure}

Not all of the operations in Figure~\ref{fig:power-standard} are defined for the
circulant algebra.  In the
first line, we use the norm $\normof{\vx\itn{0}}$ that returns
a scalar in $\RR$.  We also use the scalar inverse $\alpha^{-1}$.
The next operation is the $\sign$ function for a scalar.
Let us define these operations, in order of
their complexity.  In the next section, we will reinterpret
these operations in light of the relationships between
the fast Fourier transform and circulant matrices.
This will help illuminate a few additional properties
of these operations and will let us state an ordering for elements.

\subsection{The scalar inverse}
We begin with the scalar inverse.  Recall that
all operations between scalars behave like circulant
matrices.  Thus, the inverse of $\calpha \in \KK_k$ is
\[ \calpha^{-1} \circeq \tcirc(\calpha)^{-1}. \]
The matrix $\tcirc(\calpha)^{-1}$ is also
circulant~\cite{davis1979-circulant}.

\subsection{Scalar functions and the angle function}
Other scalar functions are also functions of a
matrix (see
\citet{Higham2008-functions-of-matrices}).
Let $f$
be a function, then
\[ f(\calpha) \circeq f(\tcirc(\calpha)) \]
where the right hand side is the same function
applied to a matrix.  (Note that it is not the function applied
to the matrix element-wise.)  


The sign function for a matrix is a special case.
As explained in~\citet{Higham2008-functions-of-matrices},
the sign function applied to a complex value is
the sign of the real-valued part.  We wish to use
a related concept that generalizes the real-valued
sign that we term ``angle.''  Given a complex value
$r e^{\ii \theta}$, then
$\tangle(r e^{\ii \theta}) = e^{\ii \theta}$.
For real or complex numbers $x$, we then have
\[ \tangle(x) \absof{x} = x. \]
Thus, we define
\[ \tangle(\calpha) \circeq \tcirc(\tabs(\calpha))^{-1} \tcirc(\calpha). \]

\subsection{Inner products, norms, and conjugates}
\label{sec:norm}

We now proceed  to define a norm.  The norm of
a vector in $\KK_k^n$ produces a scalar in $\KK_k$:
\[ \normof{\cvx} \circeq (\tcirc(\cvx)^* \tcirc(\cvx))^{1/2} =
    \left( \sum_{i=1}^n \tcirc(\cx_i)^* \tcirc(\cx_i) \right)^{1/2}. \]
For a standard vector $\vx \in \CC^{n}$, the norm
$\normof{\vx} = \sqrt{\vx^* \vx}$.  This definition, in turn,
follows from the standard inner product attached
to the vector space $\CC^{n}$.  As we shall see, our definition
has a similar interpretation.  The inner product implied by our
definition is
\[ \iprod{\cvx}{\cvy} \circeq \tcirc(\cvy)^* \tcirc(\cvx). \]
Additionally, this definition implies that that the \emph{conjugate}
operation in the circulant algebra corresponds to the transpose
of the circulant matrix
\[ \overline{\calpha} \circeq \tcirc(\calpha)^*. \]
With this conjugate, our inner product
satisfies two of the standard properties: conjugate symmetry
$\iprod{\cvx}{\cvy} = \overline{\iprod{\cvy}{\cvx}}$ and
linearity $\iprod{\calpha \circ \cvx}{\cvy} = \calpha \circ \iprod{\cvx}{\cvy}$.
The notion of positive definiteness is more intricate and we
delay that discussion until after introducing
a decoupling technique using the fast Fourier transform
in the following section.   Then, in Section~\ref{sec:triangle},
we use positive definiteness to
demonstrate a Cauchy-Schwarz inequality, which in turn
provides a triangle inequality for the norm.

\section{Operations with the fast Fourier transform}
\label{sec:fft}
In Section~\ref{sec:ops}, we explained the basic operations
of the circulant algebra as operations between matrices.
All of these matrices consisted of circulant blocks.
In this section, we show how to accelerate these operations
by exploiting the relationship between the fast Fourier transform and circulant matrices.

Let $\mC$ be a $k \times k$ circulant
matrix. Then the
eigenvector matrix of $\mC$ is given by the $k \times k$
discrete Fourier transform matrix $\mF$, where
\[ F_{ij} = \frac{1}{\sqrt{k}} \omega^{(i-1)(j-1)} \]
and $\omega = e^{2\pi \ii/k}$.
This matrix is complex symmetric, $\mF^T = \mF$, and
unitary, $\mF^* = \mF^{-1}$.  Thus, $\mC = \mF \mD \mF^*$,
$\mD = \diag(\lambda_1, \ldots, \lambda_k)$.
Recall that multiplying a vector by $\mF$ or $\mF^*$ can be
accomplished via the fast Fourier transform in
$\mathcal{O}(k \log k)$ time instead of $\mathcal{O}(k^2)$
for the typical matrix-vector product algorithm.  Also,
computing the matrix $\mD$ can be done in time $O(k \log k)$
as well.

To express our operations, we define a new transformation,
the ``Circulant Fourier Transform'' or $\cft$.  Formally,
$\fft : \calpha \in \KK_k \mapsto \CC^{k \times k}$ and its
inverse $\ifft : \CC^{k \times k} \mapsto \KK_k$ as follows:
\[ \fft(\calpha) \eqdef \sbmat{ \hat{\alpha}_1 \\ & \ddots \\ & & \hat{\alpha}_k }
      = \mF^* \tcirc(\calpha) \mF,
   \qquad
   \ifft \left( \sbmat{ \hat{\alpha}_1 \\ & \ddots \\ & & \hat{\alpha}_k } \right) \eqdef
   \calpha \circeq \mF \fft(\calpha) \mF^*, \]
where $\hat{\alpha}_j$ are the eigenvalues of $\tcirc(\calpha)$ as produced
in the Fourier transform order.  These transformations satisfy
$\ifft(\fft(\calpha)) = \calpha$ and provide a convenient way of moving
between operations in $\KK_k$ to the more familiar environment
of diagonal matrices in $\CC^{k \times k}$.

The $\cft$ and $\icft$ transformations are extended to matrices and
vectors over $\KK_k$
differently than the $\tcirc$ operation we saw before.   Observe
that $\fft$ applied ``element-wise'' to the $\tcirc(\cmA)$ matrix produces a matrix
of diagonal blocks.  In our extension of the $\fft$ routine,
we perform an additional permutation to expose
block-diagonal structure from these diagonal blocks.  This permutation
 $\mP_m \mA \mP_n^T$
transforms an $mk \times nk$ matrix of $k \times k$
diagonal blocks into a block diagonal $mk \times nk$ with $m \times n$ size
blocks.
It is also known as a
\emph{stride permutation matrix}~\cite{Granata1992-tensor}.
The construction of $\mP_m$, expressed in \textsc{Matlab} code is
\begin{quote}
\begin{verbatim}
p = reshape(1:m*k,k,m)';
Pm = sparse(1:m*k,p(:),1,m*k,m*k);
\end{verbatim}
\end{quote}
The construction for $\mP_n$ is identical.
In Figure~\ref{fig:circ-fft}, we illustrate the overall
transformation process that extends $\fft$ to matrices
and vectors.
\begin{figure}
\centering
\includegraphics[width=\linewidth]{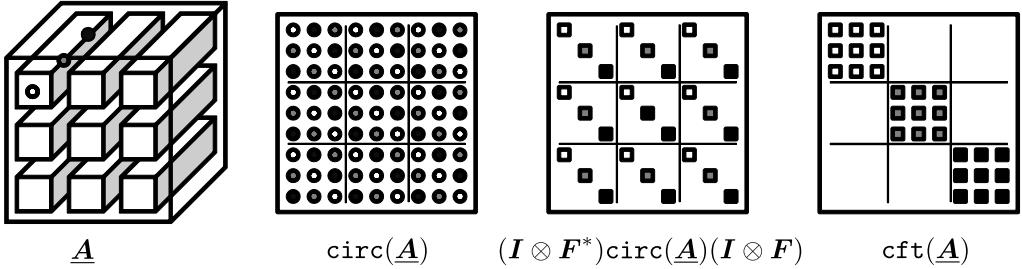}

\caption{The sequence of transformations in our $\fft$ operation.
Given a circulant $\mA$, we convert it into a matrix by $\tcirc(\mA)$.
The color of the circles in the figure is emphasizing the circulant
structure, and not equality between blocks.  In the third
figure, we diagonalize each
circulant using the Fourier transform.  The pattern of
eigenvalues is represented by squares.  Here, we are coloring
the squares to show the reordering induced by the permutation
at the final step of the $\fft$ operation.}
\label{fig:circ-fft}
\end{figure}

Algebraically, the $\fft$ operation for a matrix $\cmA \in \KK_k^{m \times n}$ is
\[ \fft(\cmA) = \mP_m (\eye_m \kron \mF^*) \tcirc(\cmA) (\eye_n \kron \mF) \mP_n^T, \]
where $\mP_m$ and $\mP_n$ are the 
permutation matrices introduced
above.
We can equivalently write this directly in terms of the
eigenvalues of each of the circulant blocks of $\tcirc(\cmA)$:
\[ \fft(\cmA) \eqdef \sbmat{ \mAhat_1 \\ & \ddots \\ & & \mAhat_k },
    \qquad
    \mAhat_j =
      \sbmat{ \lambda_j^{1,1} & \ldots & \lambda_j^{1,n} \\
              \vdots & \ddots & \vdots \\
              \lambda_j^{m,1} & \ldots & \lambda_j^{m,n} }, \]
where $\lambda_1^{r,s}, \ldots, \lambda_k^{r,s}$ are the diagonal
elements of $\fft(\cA_{r,s})$.
The inverse operation $\ifft$, takes a block diagonal matrix and returns
the matrix in $\KK_k^{m \times n}$:
\[ \ifft(\mA) \circeq (\eye_m \kron \mF) \mP_m^T \mA \mP_n (\eye_n \kron \mF^*). \]

Let us close this discussion by providing a concrete example of this operation.
\begin{example} \label{ex:cft}
Let $\cmA = \sbmat{ \csbmat{2 & 3 & 1} & \csbmat{8 & \minus2 & 0} \\
                 \csbmat{\minus2 & 0 & 2} & \csbmat{3 & 1 & 1} }$ .
The result of the $\tcirc$ and $\fft$ operations, as illustrated in Figure~\ref{fig:circ-fft}, are:
\[
\begin{aligned}
\tcirc(\cmA)
& = \left[  \begin{array}{>{\scriptstyle}c>{\scriptstyle}c>{\scriptstyle}c|>{\scriptstyle}c>{\scriptstyle}c>{\scriptstyle}c}
                     2 & 1 & 3 & 8 & 0 & \minus2 \\
                     3 & 2 & 1 & \minus2 & 8 & 0 \\
                     1 & 3 & 2 & 0 & \minus2 & 8 \\ \midrule
                     \minus2 & 2 & 0 & 3 & 1 & 1 \\
                     0 & \minus2 & 2 & 1 & 3 & 1 \\
                     2 & 0 & \minus2 & 1 & 1 & 3
               \end{array} \right],
\\
(\eye \kron \mF^*)\tcirc(\cmA)(\eye \kron \mF)
& = \left[ \begin{array}{>{\scriptstyle}c>{\scriptstyle}c>{\scriptstyle}c|>{\scriptstyle}c>{\scriptstyle}c>{\scriptstyle}c}  6 & & & 6 & & \\
                       & \minus\sqrt{3}\ii      & &   & \minus9 + \sqrt{3}\ii & \\
                       &   &  \sqrt{3}\ii & & & \minus9 - \sqrt{3}\ii \\ \midrule
                       0 &   &  & 5 & & \\
                         & \minus3 + \sqrt{3}\ii  &  &                & 2 & \\
                         &   & \minus3 - \sqrt{3}\ii &                 & &  2 \\
              \end{array} \right],
\\
\fft(\cmA)
& = \left[ \scriptstyle  \begin{array}{>{\scriptstyle}c>{\scriptstyle}c|>{\scriptstyle}c>{\scriptstyle}c|>{\scriptstyle}c>{\scriptstyle}c} 6 & 6 & & & &\\
                     0 & 5 & & & & \\  \midrule
                       &   &  \minus\sqrt{3}\ii & \minus9 + \sqrt{3}\ii \\
                       &   &  \minus3 + \sqrt{3}\ii & 2 \\  \midrule
                       &   &                   &              & \sqrt{3}\ii & \minus9 - \sqrt{3}\ii \\
                       &   &                   &              & \minus3-\sqrt{3}\ii & 2
              \end{array} \right].
\end{aligned}
\]
\end{example}

\subsection{Operations}
We now briefly illustrate how the $\fft$ accelerates
and simplifies many operations.  Let $\calpha, \cbeta \in \KK_k$.
Note that
\[
\begin{aligned}
 \calpha + \cbeta & = \ifft( \fft(\calpha) + \fft(\cbeta) ), \text{ and } \\
 \calpha \circ \cbeta & = \ifft( \fft(\calpha) \fft(\cbeta) ). \\
\end{aligned}
\]
In the Fourier space -- the output of the $\fft$ operation --
these operations are both $O(k)$ time because they occur between diagonal
matrices.
Due to the linearity of the $\fft$ operation, arbitrary sequences of
operations in the Fourier space
transform back seamlessly, for instance
\[ \underbrace{(\calpha + \cbeta) \circ (\calpha + \cbeta)  \circ \ldots \circ (\calpha + \cbeta) }_{j \text{ times}} = \ifft( (\fft(\calpha) + \fft(\cbeta))^j ). \]
But even more importantly, these simplifications generalize to matrix-based
operations too.  For example,
\[ \cmA \circ \cvx = \ifft( \fft(\cmA) \fft(\cvx) ). \]
In fact, in the Fourier space, this system is a series of
independent matrix vector products:
\[ \fft(\cmA) \fft(\cvx) =
  \sbmat{\mAhat_1 \\ & \ddots \\ & & \mAhat_k}
  \sbmat{\vxhat_1 \\ & \ddots \\ & & \vxhat_k}
  = \sbmat{\mAhat_1 \vxhat_1 \\ & \ddots \\ & & \mAhat_k \vxhat_k} .
\]
Here, we have again used $\mAhat_j$ and $\vxhat_j$ to denote the
blocks of Fourier coefficients, or equivalently, circulant eigenvalues.
\emph{The rest of the paper frequently uses this convention and
shorthand where it is clear from context.}
This formulation takes
\[ \underbrace{O(m n k \log k + n k \log k)}_{\fft \text{ and } \ifft} + \underbrace{O(k m n)}_{\text{matvecs}} \]
operations instead of $O(mnk^2)$ using the $\tcirc$ formulation
in the previous section.

More operations are simplified in the Fourier space too.
Let $\fft(\calpha) = \diag\sbmat{\hat{\alpha}_1, & \ldots, & \hat{\alpha}_k}$.
Because the $\hat{\alpha}_j$ values are the eigenvalues of $\tcirc(\alpha)$, the following functions simplify:
\[
\begin{aligned}
\tabs(\calpha) & = \ifft( \diag\sbmat{|\hat{\alpha}_1|, & \ldots,  & |\hat{\alpha}_k|}), \\
\conj{\calpha} & = \ifft( \diag\sbmat{\conj{\hat{\alpha}_1}, & \ldots,  & \conj{\hat{\alpha}_k}}) = \ifft(\fft(\calpha)^*), \text{ and } \\
\tangle(\calpha) & = \ifft( \diag\sbmat{\hat{\alpha}_1/|\hat{\alpha}_1|, & \ldots, & \hat{\alpha}_k/|\hat{\alpha}_k| } ).
\end{aligned}
\]

\paragraph{Complex values in the CFT}

A small concern with the $\ifft$ operation is that it may produce
complex-valued elements in $\KK_k$.
It suffices to note that when the output of a sequence of
circulant operations produces a real-valued circulant,
then the output of $\ifft$ is also real-valued.  In other
words, there is no problem working in Fourier space instead
of the real-valued circulant space.  This
fact can be formally verified by first formally stating the conditions
under which $\ifft$ produces real-valued circulants
($\icft(\mD)$ is real if and only if $\mF^2 \mD \mF^2 = \mD^*$,
see \citet{davis1979-circulant}),
and then
checking that the operations in the Fourier space do not
alter this condition.

\subsection{Properties}
Representations in Fourier space are convenient for illustrating
some properties of these operations. 

\begin{proposition}
The matrix $\tcirc(\tangle(\calpha))$ is orthogonal.
\end{proposition}
\begin{proof} We have
\[ \begin{aligned}
&   \tcirc(\tangle(\calpha))^* \tcirc(\tangle(\calpha))
 \circeq  \\
 & \qquad   \conj{\tangle(\calpha)} \circ \tangle(\calpha)
  =
    \ifft\left( \sbmat{
       \conj{\hat{\alpha}}_1 \hat{\alpha}_1/|\hat{\alpha}_1|^2 \\
       & \ddots
       \\ & & \conj{\hat{\alpha}}_k \hat{\alpha}_k/|\hat{\alpha}_k|^2 }
      \right)  = \cone.
\end{aligned} \]
\end{proof}
Additionally, the Fourier space is an easy place to understand
spanning sets and bases in $\KK_k^m$, as the following proposition shows.
\begin{proposition} \label{thm:basis}
Let $\cmX \in \KK_k^{m \times n}$. Then
$\cmX$ spans $\KK_k^m$ \emph{if and only if} $\tcirc(\cmX)$
and $\fft(\cmX)$ have rank $km$.  Also $\cmX$
is a basis \emph{if and only if} $\tcirc(\cmX)$
and $\fft(\cmX)$ are invertible.
\end{proposition}
\begin{proof}
First note that $\rank(\fft(\cmX)) = \rank(\tcirc(\cmX))$
because $\fft$ is a similarity transformation applied to $\tcirc$.
It suffices to show this result for $\fft(\cmX)$, then.  Now consider
$\cvy = \cmX \circ \cva$:
\[ \begin{aligned}
\fft(\cvy) & = \fft(\cmX) \fft(\cva) ; \\
\sbmat{\vyhat_1 \\ & \ddots \\ & & \vyhat_k}
& =
\sbmat{\mXhat_1 \\ & \ddots \\ & & \mXhat_k}
  \sbmat{\vahat_1 \\ & \ddots \\ & & \vahat_k}.
\end{aligned} \]
Thus, if there is a $\cvy$
that is feasible, then all $\mXhat_j \in \CC^{m \times n}$ must be rank $m$.
Conversely, if $\fft(\cmX)$ has rank $km$
then each $\mXhat_j$ must have rank $m$, and any $\cvy$
is feasible.  The result about the basis follows
from an analogous argument.
\end{proof}

\subsection{Inner products, norms, and ordering}
\label{sec:triangle}

We now return to our inner product and norm to elaborate on the
positive-definiteness and the triangle inequality.
In terms of the Fourier transform,
\[ \iprod{\cvx}{\cvy} = \ifft(\fft(\cvy)^* \fft(\cvx)). \]
If we write this in terms of the blocks of Fourier coefficients then
\[ \fft(\cvx)^* \fft(\cvy) = \sbmat{ \vyhat_1^* \vxhat^{}_1 \\ & \ddots \\ & & \vyhat_k^* \vxhat^{}_k }. \]
For $\cvy = \cvx$, each diagonal term has the form $\vxhat_j^* \vxhat^{}_j \ge 0$.
Consequently, we do consider this a positive semi-definite inner product because
the output $\tcirc(\iprod{\cvx}{\cvy})$ is a matrix with non-negative
eigenvalues.  This idea motivates the following definition of element ordering.
\begin{definition}[Ordering] \label{def:ordering}
 Let $\calpha, \cbeta \in \KK_k$.  We write
 \[
\begin{aligned}
 \calpha \le \cbeta
   \qquad \text{ when } & \qquad
   \diag( \cft(\calpha)) \le \diag(\fft(\cbeta)) \quad \text{ element-wise, and} \\
\calpha < \cbeta
   \qquad \text{ when } & \qquad
   \diag( \cft(\calpha)) < \diag(\fft(\cbeta)) \quad \text{ element-wise}.
\end{aligned}	\]
\end{definition}

We now show that our inner product satisfies the Cauchy-Schwarz
inequality:
\[ \tabs{\iprod{\cvx}{\cvy}} \le \normof{\cvx} \circ \normof{\cvy}. \]
In Fourier space, this fact holds because
$|\vyhat_j^* \vxhat_j| \le \normof{\vxhat_j} \normof{\vyhat_j}$ follows from the
standard Cauchy-Schwarz inequality.  Using this inequality,
we find that our norm satisfies the triangle inequality:
\[ \normof{\cvx + \cvy}^2 = \iprod{\cvx + \cvy}{\cvx + \cvy}
   \le \iprod{\cvx}{\cvx} + \cel{2} \circ \normof{\cvx} \circ \normof{\cvy} + \iprod{\cvy}{\cvy}
  = (\normof{\cvx} + \normof{\cvy})^2. \]
In this expression, the constant $\cel{2}$ is twice the multiplicative identify, that is $\cel{2} = \csbmat{ 2 & 0 & \ldots & 0 }$.

\section{Eigenvalues and Eigenvectors}
\label{sec:eigen}

With the results of the previous few sections, we can
now state and analyze an eigenvalue problem in
circulant algebra.  \citet{braman201x-tensor-eigenvalues} investigated
these already and proposed a decomposition approach to compute them.
We offer an extended analysis that addresses a few additional aspects.
Specifically, we focus on a {\em canonical} set of eigenpairs.

Recall that eigenvalues of matrices are the roots of the
characteristic polynomial $\det(\mA - \lambda \eye) = 0.$
Now let $\cmA \in \KK_k^{n \times n}$
and $\clambda \in \KK_k$.
The eigenvalue problem does not change:
\[ \det(\cmA - \cel{\lambda} \circ \cmI) = \cel{0}. \]
(As an aside, note that the standard properties of the determinant hold for any
matrix over a commutative ring with identity; in particular,
the Cayley-Hamilton theorem holds in this algebra.)
The existence of an eigenvalue implies
the existence of a corresponding eigenvector $\cvx \in \KK_k^{n}$.
Thus, an eigenvalue and eigenvector pair in this algebra is
\[ \cmA \circ \cvx = \clambda \circ \cvx. \]

Just like the matrix case, these eigenvectors can
be rescaled by any constant $\calpha \in \KK_k$:
$ \cmA \circ \calpha \circ \cvx = \clambda \circ \calpha \circ \cvx. $
In terms of normalization, note that
$\normof{\cbeta \circ \cvx} = \normof{\cvx}$
if $\tcirc(\cbeta)$ is an orthogonal circulant.  This follows
most easily by noting that
\[ \normof{\cbeta \circ \cvx} \circeq
   \left( \sum_{i=1}^n \tcirc(\cbeta)^* \tcirc(\vx_i)^* \tcirc(\vx_i) \tcirc(\cbeta) \right)^{1/2}
   \circeq \normof{\cvx}, \]
because circulant matrices commute and $\tcirc(\cbeta)$ is orthogonal by construction.
For this reason, we consider orthogonal circulant matrices the
analogues of \emph{angles} or \emph{signs}, and normalized eigenvectors in the circulant algebra
can be rescaled by them. (Recall that we showed
that $\tangle(\calpha)$ is an orthogonal circulant in
Section~\ref{sec:fft}.)


The Fourier transform offers a
convenient decoupling procedure to compute eigenvalues and eigenvectors,
as observed by \citet{braman201x-tensor-eigenvalues}.
 Let $\cmA \in \KK_k^{n \times n}$ and let
$\cvx \in \KK_k^n$ and $\clambda$ be an eigenvalue
and eigenvector pair: $\cmA \circ \cvx = \cvx \circ \clambda$
and $\det(\cmA - \clambda \circ \cmI) = 0$.
Then it is straightforward to show that the Fourier transforms $\fft(\cmA)$, $\fft(\cvx)$,
and $\fft(\clambda)$ decouple as follows:
\[\begin{aligned}
 \fft(\cmA \circ \cvx) & = \fft(\cvx \circ \clambda) ; \\
 \fft(\cmA) \fft(\cvx) & = \fft(\cvx) \fft(\clambda) ; \\
  \sbmat{\mAhat_1 \\ & \ddots \\ & & \mAhat_k}
  \sbmat{\vxhat_1 \\ & \ddots \\ & & \vxhat_k}
  & = \sbmat{\vxhat_1 \\ & \ddots \\ & & \vxhat_k}
    \sbmat{ \hat{\lambda}_1 \\ & \ddots \\ & & \hat{\lambda}_k } ;
  \\
 \sbmat{\mAhat_1 \vxhat_1 \\ & \ddots \\ & & \mAhat_k \vxhat_k}
 & = \sbmat{ \hat{\lambda}_1 \vxhat_1 \\ & \ddots \\ & & \hat{\lambda}_k \vxhat_k },
  \end{aligned}
\]
where $\hat{\lambda}_j \in \lambda(\mAhat_j)$ and $\vxhat_j \not= 0$.
The last equation follows because
\[\fft(\det(\cmA - \clambda \circ \cI)) =
   \diag\sbmat{\det(\mAhat_1 - \hat{\lambda}_1 \eye), & \ldots, & \det(\mAhat_k - \hat{\lambda}_k \eye)}
   = 0.
\]

The decoupling procedure we just described shows that \emph{any}
eigenvalue or eigenvector of $\cmA$ must
decompose into individual eigenvalues or eigenvectors of the
$\fft$-transformed problem.
This illustrates a fundamental difference from the
standard matrix algebra. For standard matrices, requiring
$\det(\mA - \lambda \eye) = 0$ and finding a nonzero solution
$\vx$ for $\mA \vx=\lambda \vx$ are equivalent. In contrast,
the determinant and the eigenvector
equations are not  equivalent in the circulant algebra:
$\cmA \circ \cvx = \cvx \circ \clambda$ actually has an
infinite number of solutions $\clambda$.  For instance,
set $\vxhat_1, \hat{\lambda}_1$
to be an eigenpair of $\mAhat_1$ and $\vxhat_j=0$ for $j>1$, then any
value for $\hat{\lambda}_j$ solves $\cmA \circ \cvx = \cvx \circ \clambda$.
However, only a few of these solutions also satisfy
$\det(\cmA - \clambda \circ \cmI)=\cel{0}$.

Eigenvalues of matrices in $\KK_k^{n \times n}$ have some
interesting properties. Most notably, a matrix may have more than $n$
eigenvalues. As a special case, the diagonal elements of a
matrix are not necessarily the only eigenvalues.  We demonstrate
these properties with an example.
\begin{example} \label{ex:diag-evals}
 For the diagonal matrix
  \begin{equation*}
    \bmat{ \csbmat{2 & 3 & 1} & \csbmat{0 & 0 & 0} \\
           \csbmat{0 & 0 & 0} & \csbmat{3 & 1 & 1} }
  \end{equation*}
  we have
  \[ \mAhat_1 = \bmat{6 & 0\\0 & 5}, \quad
   \mAhat_2 = \bmat{\minus\ii \sqrt{3} & 0 \\ 0 & 2}, \quad
   \mAhat_3 = \bmat{\ii \sqrt{3} & 0 \\ 0 & 2}. \]

%
%
%

Thus,
\[
\begin{aligned}
\clambda_1 & = \ifft(\diag\sbmat{6 & 2 & 2}) = \smallmath{(1/3)} \csbmat{10 & 4 & 4}&
 \qquad
\clambda_2 & = \ifft(\diag\sbmat{5 & \minus\ii \sqrt{3} & \ii\sqrt{3}}) = \smallmath{(1/3)}\csbmat{5 & 2 & 2} \\
\clambda_3 & = \ifft(\diag\sbmat{6 & \minus\ii \sqrt{3} & \ii\sqrt{3}}) = \csbmat{2 & 3 & 1} &
\qquad
\clambda_4 & = \ifft(\diag\sbmat{5 & 2 & 2}) = \smallmath{(1/3)}\csbmat{3 & 1 & 1}. \\
\end{aligned}
 \]
 The corresponding eigenvectors are
  \[
  \begin{array}{l@{\qquad}l}
       \cvx_1 = \bmat{\csbmat{ 1/3 & 1/3 & 1/3} \\ \csbmat{ 2/3 & \minus1/3 &\minus1/3} } ;
     &
      \cvx_2 = \bmat{\csbmat{ 2/3 & \minus 1/3 & \minus 1/3 } \\ \csbmat{1/3 & 1/3 & 1/3} } ;
     \\[3ex]
       \cvx_3 = \bmat{\csbmat{1 & 0 & 0} \\ \csbmat{0 & 0 & 0}} ;
     &
      \cvx_4 = \bmat{\csbmat{0 & 0 & 0} \\ \csbmat{1 & 0 & 0}}.
  \end{array}
  \]

There are still more eigenvalues, however.  The four eigenvalues above all
correspond to elements in $\KK_k$ with real-valued entries.
We can combine the eigenvalues of the $\mAhat_j$'s to produce
complex-valued elements in $\KK_k$ that are also eigenvalues.
These are
\[
\begin{aligned}
\clambda_5 & = \ifft(\diag\sbmat{6 & \minus\ii\sqrt{3} & 2}) &
 \qquad
\clambda_6 & = \ifft(\diag\sbmat{6 & 2 & \ii\sqrt{3}}) \\
\clambda_7 & = \ifft(\diag\sbmat{5 & \minus\ii \sqrt{3} & 2}) &
\qquad
\clambda_8 & = \ifft(\diag\sbmat{5 & 2 & \ii \sqrt{3}}). \\
\end{aligned}
 \]

For completeness and further clarity, let us extend this example a bit by presenting also the eigenvalues of the non-diagonal matrix from Example~\ref{ex:cft}.
Let $\cmA = \sbmat{ \csbmat{2 & 3 & 1} & \csbmat{8 & \minus2 & 0} \\
                 \csbmat{\minus2 & 0 & 2} & \csbmat{3 & 1 & 1} }$ .
The $\cft$ produces:
\[ \mAhat_1 = \bmat{6 & 6\\0 & 5}, \quad
   \mAhat_2 = \bmat{\minus \sqrt{3} & \minus9 + \ii\sqrt{3} \\ \minus3 + \ii\sqrt{3} & 2}, \quad
   \mAhat_3 = \bmat{\ii \sqrt{3} & \minus9 - \ii\sqrt{3} \\ \minus3 + \ii\sqrt{3} & 2}. \]
The numerical eigenvalues of $\mAhat_1$ are $\{6, 5\}$; of $\mAhat_2$ are
$\{ \minus0.0899 + 6.4282\ii, 2.0899 - 4.6962\ii \}$; and of
$\mAhat_3$ are $\{ \minus0.0899 - 6.4282\ii, 2.0899 + 4.6962\ii \}$.
The real-valued eigenvalues of $\cmA$ are
\[
\begin{aligned}
\clambda_1 & = \csbmat{1.9401 & \minus1.6814 & 5.7413} &
\clambda_2 & = \csbmat{3.0599 & 3.6814 & \minus1.7413}  \\
\clambda_3 & = \csbmat{3.3933 & 4.0147 & \minus1.4080} &
\clambda_4 & = \csbmat{1.6067 & \minus2.0147 & 5.4080} .
\end{aligned}
\]
The complex-valued eigenvalues of $\cmA$ are 
\[
\begin{aligned}
\clambda_5 & = \csbmat{4.6966 - 1.5654\ii, & \minus0.7040 + 1.9114\ii, & 2.0073 - 0.3461\ii} \\
\clambda_6 & = \csbmat{3.6367 + 2.1427\ii, & 3.0373 + 0.3980\ii, & \minus0.6740 - 2.5407\ii}  \\
\clambda_7 & = \csbmat{4.3633 - 1.5654\ii, & \minus1.0373 + 1.9114\ii, & 1.6740 - 0.3461\ii } \\
\clambda_8 & = \csbmat{3.3034 + 2.1427\ii, & 2.7040 + 0.3980\ii, & \minus1.0073 - 2.5407\ii} .
\end{aligned}
\]
\end{example}

We now count the number of unique eigenvalues and eigenvectors,
using the decoupling procedure in the Fourier space.
To simplify the discussion, let us only consider the case
where each $\mAhat_j$ has simple eigenvalues.
Consider an $\cmA \in \KK_k^{n \times n}$ with this property,
and let $m_j$ be the number of unique eigenvalues
and eigenvectors of $\mAhat_j$.
Then the number
of unique eigenvalues of $\cmA$ is given by the number of
unique solutions to $\det(\cmA - \clambda \circ \cmI)=0$
 which is $\prod_{j=1}^k m_j$.  The number of unique
 eigenvectors (up to normalization) is given by the number
 of unique solutions to $\cmA \circ \cvx = \cvx \circ \clambda$,
 which is also $\prod_{j=1}^k m_j$.

This result shows
there are at most $n^k$ eigenvalues if $\clambda \in \KK_k$
is allowed to be complex-valued, even when $\cmA \in \KK_k$
is real-valued.  If $\cmA \in \KK_k$ is real-valued,
then there are at most $n^{\ceilof{(k+1)/2}}$ ``real''
eigenvalues.  For this result,
note that $\ifft(\diag\sbmat{\alpha_1 & \ldots & \alpha_k})$ is real-valued
if and only if $\diag\sbmat{\alpha_1 & \ldots & \alpha_k}^* =
\fftm^2 \diag\sbmat{\alpha_1 & \ldots & \alpha_k}\fftm^2$~\cite{davis1979-circulant},
where $\fftm$ is the Fourier transform matrix.   This
implies $\alpha_1$ is real-valued, and
$\alpha_j = \conj{\alpha_{k-j+1}}$.  Applying this restriction
reduces the feasible combinations
of eigenvalues to $n^{\ceilof{(k+1)/2}}$.

Given that there are so many eigenvalues and vectors, are all
of them necessary to describe $\cmA$?  We now show this is
not the case by making a few definitions to clarify
the discussion.
\begin{definition}
 Let $\cmA \in \KK_k^{n \times n}$.
  A  \emph{canonical} set
 of eigenvalues and eigenvectors is a set of minimum
 size, ordered such that
 $\tabs(\clambda_1) \ge \tabs(\clambda_2) \ge \ldots \ge \tabs(\clambda_k)$,
  which contains the information to
  reproduce \emph{any} eigenvalue or eigenvector of $\cmA$
\end{definition}
In the diagonal matrix from Example~\ref{ex:diag-evals},
%
the sets $\{ (\clambda_1, \cvx_1), (\clambda_2, \cvx_2) \},$
 $\{ (\clambda_3, \cvx_3), (\clambda_4, \cvx_4) \},$ and
$\{ (\clambda_1, \cvx_1), (\clambda_3, \cvx_3), (\clambda_4, \cvx_4) \}$
contain all the information to reproduce any eigenpair, whereas
the set $\{ (\clambda_1, \cvx_1), (\clambda_3, \cvx_3) \}$
does not (it does not contain the eigenvalue $5$ of $\mAhat_1$).
In this case,
the only canonical set is $\{ (\clambda_1, \cvx_1), (\clambda_2, \cvx_2) \}$.
This occurs because, by a simple counting argument, a canonical
set must have at least two eigenvalues, thus the set is of minimum size.
The choice of $\clambda_1$ and $\clambda_2$ is given by the ordering
condition.  Among all the size $2$ sets with all the information,
this is the only one with the property that $\tabs(\clambda_1) \ge \tabs(\clambda_2)$.

\begin{theorem}[Unique Canonical Decomposition] \label{thm:canonical}
 Let $\cmA \in \KK_k^{n \times n}$ where each $\mAhat_j$ in
 the $\fft(\cmA)$ matrix has distinct eigenvalues with distinct
 magnitudes.
 Then $\cmA$ has a unique canonical set of $n$ eigenvalues
 and eigenvectors.
 This canonical set corresponds to a basis of $n$ eigenvectors, yielding
 an eigendecomposition \[ \cmA = \cmX \circ \cmLambda \circ \cmX^{-1} . \]
\end{theorem}
\begin{proof}
Because all of the eigenvalues of
each $\mAhat_j$ are distinct, with distinct magnitudes, there
are $nk$ distinct numbers.  This implies that any canonical
set must have at least $n$ eigenvalues.

Let $\hat{\lambda}_j^{(i)}$ be the $i$th eigenvalue of $\mAhat_j$
ordered such that $|\hat{\lambda}_j^{(1)}| > |\hat{\lambda}_j^{(2)}| >
\ldots > |\hat{\lambda}_j^{(n)}|$.  Then $\clambda_i = \icft (
\diag \bmat{ \hat{\lambda}_1^{(i)}, & \ldots, & \hat{\lambda}_k^{(i)} } )$
is a canonical set of eigenvalues.  We now show that this set
constitutes an eigenbasis.
Let $\mAhat_j = \mXhat_j \mLambdahat_j \mXhat_j^{-1}$ be
the eigendecomposition using the magnitude ordering above.
Then $\cmX = \icft( \diag \bmat{ \mXhat_1, & \ldots, & \mXhat_k} )$
and $\cmLambda = \icft( \diag \bmat{ \mLambdahat_k, & \ldots, & \mLambdahat_k} )$
is an eigenbasis because the matrix $\cmX$ satisfies the
properties of a basis from Theorem~\ref{thm:basis}.
Note that $\cmLambda_{i,i} = \clambda_i$.

Finally, we show that the set is unique. In any canonical
set $\tabs(\clambda_1) \ge \tabs(\clambda_i)$ for $i > 1$.
In the Fourier space, this implies
$|\hat{\lambda}_j^{(1)}| \ge |\hat{\lambda}_j^{(i)}|$.  Because
all of the values $|\hat{\lambda}_j^{(i)}|$ are unique, there
is no choice for $\hat{\lambda}_j^{(1)}$ in a canonical set
and we have $|\hat{\lambda}_j^{(1)}| > |\hat{\lambda}_j^{(i)}|, i > 1$.
Consequently, $\clambda_1$ is unique.  Repeating this
argument on the remaining choices for $\clambda_i$
shows that the entire set is unique.
\end{proof}
\begin{remark} If $\mAhat_j$ has distinct eigenvalues but they do not have distinct
magnitudes, then $\cmA$ has an eigenbasis but the canonical
 set may not be unique, because $\mAhat_j$ may have two distinct eigenvalues
 with the same magnitude.
\end{remark}

Next, we show that the eigendecomposition is \emph{real-valued}
under a surprisingly mild condition.
\begin{theorem} \label{thm:real-eigendecomposition}
 Let $\cmA \in \KK_k^{n \times n}$ be real-valued with
 diagonalizable $\mAhat_j$ matrices.  If $k$ is odd, then
 the eigendecomposition
 $\cmX \circ \cmLambda \circ \cmX^{-1}$ is
 real-valued if and only if $\mAhat_1$ has real-valued
 eigenvalues.  If $k$ is even, then
 $\cmX \circ \cmLambda \circ \cmX^{-1}$ is
 real-valued if and only if $\mAhat_1$ and $\mAhat_{k/2 + 1}$
 have real-valued eigenvalues.
\end{theorem}
\begin{proof}
First, if $\cmA$ has a real-valued
eigendecomposition, then we have that
$\mXhat_1$ is real and also that $\mXhat_{k/2+1}$
is real when $k$ is even.  Likewise,
$\mLambdahat_1$ is real and $\mLambdahat_{k/2+1}$
is real when $k$ is even.  Thus, $\mAhat_1$ (and
also $\mAhat_{k/2+1}$ when $k$ is even) have real-valued
eigenvalues and vectors.

When $\mAhat_1$ (and $\mAhat_{k/2+1}$ for $k$ even)
have real-valued eigenvalues and vectors, then note that
we can choose eigenvalues and eigenvectors of the
other matrices $\mAhat_j$, which may be complex, in
complex-conjugate pairs so as to satisfy the condition
for a real-valued inverse Fourier transforms.  This
happens because when $\cmA$ is real, then
$\mAhat_1$ is real and
 $\mAhat_j = \conj{\mAhat_{k-j+2}}$ by the properties
 of the Fourier transform \cite{davis1979-circulant}.
 Thus for each eigenpair $\hat{\lambda}_j, \vxhat_j$ of
 $\mAhat_j$, the  pair $\conj{\hat{\lambda}}_j,
 \conj{\vxhat}_j$ is an eigenpair for $\mAhat_{k-j+2}$.
 Consequently, if we always choose these complex conjugate pairs
 for all $j$ besides $j=1$ (and $j=k/2+1$ for $k$ even), then
 the result of the inverse Fourier transform will be real-valued.
\end{proof}

Finally, we note that if the scalars of a matrix are
padded with zeros to transform them into the circulant algebra, then 
the canonical set of eigenvalues are nothing but tuples that consist of 
the eigenvalues of the original matrix in the first entry, 
padded with zeros as well. To justify this observation, let 
$\cmA \in \KK_k^{n \times n}$ have $\cel{A}_{i,j}
= \csbmat{G_{i,j}, & 0, & \ldots, & 0}$
for a  matrix $\mG \in \mathbb{R}^{n \times n}$.
Also, let $\lambda_1, \ldots, \lambda_m \quad (m \le n)$ be the eigenvalues of
$\mG$ ordered such that
$|\lambda_1| \ge |\lambda_2| \ge \cdots \ge |\lambda_m|$.
Then
$\cft(\cel{A}_{i,j}) = \diag\sbmat{ G_{i,j}, & \ldots, & G_{i,j} }$
and thus $\mAhat_{j} = \mG$ for all $j$.  Thus, we only need to combine
the same $m$ eigenvalues of each $\mAhat_j$ to construct eigenvalues of
$\cmA$.  For the eigenvalues $\clambda_i$, we have $\cft(\clambda) =
\diag\sbmat{ \lambda_i, & \ldots, & \lambda_i }$, thus the given set
is  canonical because of the same argument used
in the proof of Theorem~\ref{thm:canonical}.


We end this section by noting that much of the above analysis can be generalized
to non-simple eigenvalues and vectors using
the Jordan canonical form of the $\mAhat_j$ matrices.

\section{The power method and the Arnoldi method}
\label{sec:algs}

In what follows, we show that the power method in
the circulant algebra computes the eigenvalue $\clambda_1$
in the canonical set of eigenvalues.  This result shows
how the circulant algebra matches the behavior of
the standard power method.   As part of our analysis, we show
that the power method decouples into $k$ independent
power iterations in Fourier space and is equivalent
to a subspace iteration method.  Second,
we demonstrate the Arnoldi method in the circulant algebra.
In Fourier space, the Arnoldi method is also equivalent to
the Arnoldi algorithm on independent problems, and it also corresponds
to a particular block Arnoldi procedure.

\subsection{The power method}
\label{sec:power}

Please see the left half of Figure~\ref{fig:power-circulant}
for the sequence of operations in the power method in the
circulant algebra.
In fact, it is not too different from the standard
power method in Figure~\ref{fig:power-standard}.  We replace
$\mA \vx$ with $\cmA \circ \cvx$ and use the norm and
inverse from Section~\ref{sec:ops}. We'll return
to the convergence criteria shortly.
As we show next, the algorithm runs
$k$ independent power methods in Fourier space.  Thus,
the right half of Figure~\ref{fig:power-circulant} shows
the equivalent operations in Fourier space.

\begin{figure}

\begin{minipage}[t]{0.45\textwidth}
\begin{algorithmic}[1]
\REQUIRE  $\cmA, \cvx\itn{0}, \tau$
\STATE \COMMENT{Kept for alignment} $\vphantom{\mAhat \leftarrow \fft(\cmA), \mXhat\itn{0} \leftarrow \fft(\cvx\itn{0})}$
\STATE $\cvx\itn{0} \leftarrow \cvx\itn{0}\circ\normof{\cvx\itn{0}}^{-1}$
   $\vphantom{\left( { \mXhat\itn{0}}^* \mXhat\itn{0} \right)^{-1/2}}$
\FOR {$k=1, \ldots$ until convergence}
  \STATE $\cvy\itn{k} \leftarrow \cmA \circ \cvx\itn{k-1}$
      $\vphantom{\mYhat\itn{k} \leftarrow \mAhat \mXhat\itn{k-1}}$
  \STATE $\calpha\itn{k} \leftarrow \normof{\cvy\itn{k}}$
      $\vphantom{\mRhat\itn{k} \leftarrow {\mYhat\itn{k}}^* \mYhat\itn{k}}$
  \STATE $\cvx\itn{k} \leftarrow \cvy\itn{k} \circ {\calpha\itn{k}}^{-1}$
     \IF {converged}
    \RETURN $\cvx\itn{k}$ $\vphantom{\mXhat\itn{k}}$
  \ENDIF
\ENDFOR
\end{algorithmic}
\end{minipage}
\quad
\begin{minipage}[t]{0.5\textwidth}
\begin{algorithmic}[1]
\REQUIRE  $\cmA, \cvx\itn{0}, \tau$
\STATE $\mAhat \leftarrow \fft(\cmA), \mXhat\itn{0} \leftarrow \fft(\cvx\itn{0})$
\STATE $\mXhat\itn{0} \leftarrow \mXhat\itn{0}
             \left( { \mXhat\itn{0}}^* \mXhat\itn{0} \right)^{-1/2}$
\FOR {$k=1, \ldots$ until convergence}
  \STATE $\mYhat\itn{k} \leftarrow \mAhat \mXhat\itn{k-1}$
      $\vphantom{\cvy\itn{k} \leftarrow \cmA \circ \cvx\itn{k-1}}$
  \STATE \label{alg:block-power-norm-1}
    $\mRhat\itn{k} \leftarrow {\mYhat\itn{k}}^* \mYhat\itn{k}$
  \STATE \label{alg:block-power-norm-2}
  $\mXhat\itn{k} \leftarrow \mYhat\itn{k} {\mRhat\itn{k}}^{-1/2}$
    \IF {converged}
    \RETURN $\ifft(\mXhat\itn{k})$
  \ENDIF
\ENDFOR
\end{algorithmic}
\end{minipage}
\caption{The power method in the circulant algebra (left)
and the power method in the circulant algebra after
transformation with the fast Fourier transform (right).
We address convergence criteria in Section~\ref{sec:power}}.
\label{fig:power-circulant}
\end{figure}

To analyze the power
method, consider the key iterative operation in the power method when
transformed into Fourier space:
\[ \begin{aligned}
& \fft( \cmA \circ \cvx \circ (\normof{\cmA \circ \cvx})^{-1} ) \\
& \qquad  = \fft(\cmA) \fft(\cvx) (\fft(\cvx)^* \fft(\cvx))^{-1/2}  \\
& \qquad = \sbmat{\mAhat_1\vxhat_1 \\ & \ddots \\ & & \mAhat_k \vxhat_k}
  \left( \sbmat{\mAhat_1\vxhat_1 \\ & \ddots \\ & & \mAhat_k \vxhat_k}^*
   \sbmat{\mAhat_1\vxhat_1 \\ & \ddots \\ & & \mAhat_k \vxhat_k} \right)^{-1/2}.\\
  \end{aligned} \]
Now,
\[
\begin{aligned}
\left(
  \sbmat{\mAhat_1\vxhat_1 \\ & \ddots \\ & & \mAhat_k \vxhat_k}^*
  \sbmat{\mAhat_1\vxhat_1 \\ & \ddots \\ & & \mAhat_k \vxhat_k}
\right)^{-1/2}
& =
\sbmat{ \vxhat_1^* \mAhat_1^* \mAhat_1 \vxhat_1 \\
        & \ddots \\ & & \vxhat_k^* \mAhat_k^* \mAhat_k \vxhat_k }^{-1/2}
\\ & =
\sbmat{ \normof{\mAhat_1 \vxhat_1}^{-1} \\
        & \ddots \\ & & \normof{\mAhat_k \vxhat_k}^{-1} }.
\end{aligned}
\]
Thus
\[ \fft( \cmA \circ \cvx \circ (\normof{\cmA \circ \cvx})^{-1} )
= \sbmat{\mAhat_1 \vxhat_1 / \normof{\mAhat_1 \vxhat_1} \\
         & \ddots \\
         & & \mAhat_k \vxhat_k / \normof{\mAhat_1 \vxhat_1}}.
\]
The key iterative operation, $\cmA \circ \cvx \circ (\normof{\cmA \circ \cvx})^{-1}$,
 corresponds to  \emph{one step} of the standard power method on each
matrix $\mAhat_j$.  From this derivation, we arrive at the following theorem, whose 
proof follows immediately from the convergence proof of the power method for a matrix.

\begin{theorem}
 Let $\cmA \in \KK_k^{n \times n}$ have a canonical
 set of eigenvalues $\clambda_1, \ldots, \clambda_n$
 where $|\clambda_1| > |\clambda_2|$, then
 the power method in the circulant algebra
 convergences to an eigenvector $\cvx_1$
 with eigenvalue $\clambda_1$.
\end{theorem}

A bit tangentially, an eigenpair
in the Fourier space is a simple instance of
a \emph{multivariate} eigenvalue
problem~\cite{Chu1993-MEP}.
The general multivariate eigenvalue problem is
$\sum_{j} \mA_{i,j} \vx_j = \lambda_i \vx_i \qquad i=1,\ldots$,
whereas we study the same system, albeit diagonal.
\citet{Chu1993-MEP} did study a power method for the more
general problem and showed local convergence;
however our diagonal situation is sufficiently
simple for us to state stronger results.


\paragraph{Convergence Criteria}
A simple measure
such as $\normof{\cvx\itn{k}-\cvx\itn{k-1}} \le \cel{\tau}$,
with $\cel{\tau} = \csbmat{\tau & 0 & \ldots & 0}$ will not detect
convergence.  As mentioned in the description of the standard
power method in Figure~\ref{fig:power-standard}, this test
can fail when the eigenvector changes angle.  Here, we have
the more general notion of an angle for each element,
and eigenvectors are unique up to a choice of angle.  Thus,
we first normalize angles before comparing the
we use the convergence criteria 
\begin{equation} \label{eq:power-convergence}
 \normof{ \tangle(\cvx_1\itn{k})^{-1} \circ \cvx\itn{k}
     - \tangle(\cvx\itn{k-1}_1)^{-1} \circ \cvx\itn{k-1} } < \cel{\tau}.
\end{equation}
In the Fourier space, this choice requires that \emph{all}
of the independent problems have converged to a
tolerance of $\tau$, which is a viable practical choice.
An alternative convergence
criteria is to terminate when the \emph{eigenvalue}
stops changing, although this may occur significantly
before the eigenvector has converged.

\paragraph{Subspace iteration}
We now show that the power method is equivalent
to subspace iteration in Fourier space.  Subspace
iteration is also known as
``orthogonal iteration'' or the ``block-power method.''
Given a starting
block of vectors $\mX\itn{0}$, the iteration is
\[ \mY \leftarrow \mA \mX\itn{k}, \qquad
     \mX\itn{k+1}, \mR\itn{k+1} = \text{\texttt{qr}}(\mY).
\]
On the surface, there is nothing to relate this iteration to
our power method, even in Fourier space.  The relationship,
however, follows because all of our operations in
Fourier space occur with \emph{block-diagonal} matrices.
Note that for a block-diagonal matrix of vectors,
which is what $\mXhat\itn{k}$ is, the QR factorization
just normalizes each column.  In other words, the
result is a diagonal matrix $\mR$.  This simplification
shows that steps
\ref{alg:block-power-norm-1}-\ref{alg:block-power-norm-2}
in the Fourier space algorithm are equivalent to
the QR factorization in subspace iteration.

\paragraph{Breakdown}

One problem with this
iterative approach is that it can
encounter ``zero divisors'' as scalars when
running these algorithms.  These occur when the
matrices in Fourier space are not invertible.
We have not explicitly
addressed this situation and note that the same issues arise
in block methods when some of the quantities
become singular.  The analogy
with the block method may provide an appropriate
solution.  For example, if the scalar $\calpha\itn{k}$
is a zero-divisor, then we could use the QR factorization
of $\mYhat\itn{k}$ -- as suggested by the equivalence
with subspace iteration -- instead. 

\subsection{The Arnoldi process}

The Arnoldi method is a cornerstone of modern matrix computations.
Let $\mA$ be an $n \times n$ matrix with real valued entries.
Then the Arnoldi method is a technique to build an orthogonal
basis for the Krylov subspace
\[
 \mathcal{K}_t(\mA,\vv) = \tspan\{ \vv, \mA \vv, \ldots, \mA^{t-1} \vv \} ,
\]
where $\vv$ is an initial vector.
Instead of using this power basis, the Arnoldi
process derives a set of orthogonal vectors that span the same space
when computed with exact arithmetic.  The standard method is
presented in Figure~\ref{fig:arnoldi-circulant}(a).
From this procedure, we have the Arnoldi decomposition of a matrix:
\[ \mA \mQ_t = \mQ_{t+1} \mH_{t+1,t} \]
where $\mQ_t$ is an $n \times t$ matrix, and $\mH_{t+1,t}$ is a $(t+1) \times t$
upper Hessenberg matrix.
Arnoldi's orthogonal subspaces $\mQ$ enable efficient algorithms for both
solving large scale linear systems~\cite{Krylov1931-equations} and computing eigenvalues and
eigenvectors \cite{Arnoldi1951-minimized}.

Using our repertoire of operations, the Arnoldi
method in the circulant algebra is presented in
Figure~\ref{fig:arnoldi-circulant}(b).  The circulant Arnoldi process decoupled
via the $\cft$ is also shown in Figure~\ref{fig:arnoldi-circulant}(c).

We make three observations here.
First, the decoupled ($\cft$) circulant Arnoldi process
is equivalent to individual Arnoldi processes on each matrix $\mAhat_j$.
This follows by a similar analysis used to
show the decoupling result about the power method.
The verification
of this fact for the Arnoldi iteration is a bit more tedious and thus
we omit this analysis.
Second, the same decoupled process
is equivalent to a block Arnoldi process.
This also follows for the same
reason the equivalent result held for the power method: the QR
factorization of a block-diagonal matrix-of-vectors is just
a normalization of each vector.
Third,
we produce an Arnoldi factorization:
\[ \cmA \circ \cmQ_{t} = \cmQ_{t+1} \circ \cmH_{t+1,t}. \]
In fact, this outcome is a corollary of the first property
and follows from applying $\icft$ to the same analysis.

\begin{figure}

\begin{minipage}[t]{0.31\linewidth}
(a) Arnoldi for $\RR^{n \times n}$

 \begin{algorithmic}[1]
   \REQUIRE  $\mA, \vb, t$
   \STATE $\vphantom{\mAhat \leftarrow \cft(\cmA)}$
   \STATE $\vphantom{\mBhat \leftarrow \cft(\cvb)}$
   \STATE $\vq_1 \leftarrow \vb/\normof{\vb}$
	   $\vphantom{\mQhat_1 \leftarrow \mBhat (\mBhat^* \mBhat)^{-1/2}}$
   \FOR {$j = 1, \ldots, t$}
     \STATE $\vz \leftarrow \mA \vq_1$
     \FOR {$i = 1, \ldots, i$}
       \STATE $H_{i,j} \leftarrow \vq_i^* \vz$
			   $\vphantom{\cel{H}_{i,j} \leftarrow \iprod{\cvq_i}{\cvz}}$
				 $\vphantom{\mHhat_{i,j} \leftarrow \mQhat_i^* \mZhat}$
			 \STATE $\vphantom{\cmH_{i,j} \leftarrow \icft(\mHhat_{i,j})}$
       \STATE $\vz \leftarrow \vz - H_{i,j} \vq_i$
			  $\vphantom{\cvz \leftarrow \cvz - \cel{H}_{i,j} \cvq_i}$
				$\vphantom{\mZhat \leftarrow \mZhat - \mQhat_i \mHhat_{i,j}}$
     \ENDFOR
     \STATE $H_{j+1,j} \leftarrow \normof{\vz}$
		   $\vphantom{\mHhat_{j+1,j} \leftarrow (\mZhat^* \mZhat)^{1/2}}$
     \STATE $\vphantom{\cmH_{j+1,j} \leftarrow \icft(\mHhat_{j+1,j})}$
     \STATE $\vq_{j+1} \leftarrow \cvz / H_{j+1,j}$
		   $\vphantom{\mQhat_{j+1} \leftarrow \mZhat \mHhat_{j+1,j}^{-1}}$
			 $\vphantom{\cvq_{j+1} \leftarrow \cvz \circ \cel{H}_{j+1,j}^{-1}}$
     \STATE $\vphantom{\cvq_{j+1} \leftarrow \icft(\mQhat_{j+1})}$
   \ENDFOR
 \end{algorithmic}
\end{minipage}
 \begin{minipage}[t]{0.31\linewidth}
(b) Arnoldi for $\KK_k^{n \times n}$

\begin{algorithmic}[1]
\REQUIRE  $\cmA, \cvb, t$
\STATE $\vphantom{\mAhat \leftarrow \cft(\cmA)}$
\STATE $\vphantom{\mBhat \leftarrow \cft(\cvb)}$
\STATE $\cvq_1 \leftarrow \vb \circ \normof{\cvb}^{-1}$
  $\vphantom{\mQhat_1 \leftarrow \mBhat (\mBhat^* \mBhat)^{-1/2}}$
\FOR {$j = 1, \ldots, t$}
  \STATE $\cvz \leftarrow \cmA \circ \cvq_j$
  \FOR {$i = 1, \ldots, j$}
    \STATE $\cel{H}_{i,j} \leftarrow \iprod{\cvq_i}{\cvz}$
		  $\vphantom{\mHhat_{i,j} \leftarrow \mQhat_i^* \mZhat}$
    \STATE $\vphantom{\cmH_{i,j} \leftarrow \icft(\mHhat_{i,j})}$
    \STATE $\cvz \leftarrow \cvz - \cel{H}_{i,j} \circ \cvq_i$
		  $\vphantom{\mZhat \leftarrow \mZhat - \mQhat_i \mHhat_{i,j}}$
  \ENDFOR
  \STATE $\cel{H}_{j+1,j} \leftarrow \normof{\cvz}$
	  $\vphantom{\mHhat_{j+1,j} \leftarrow (\mZhat^* \mZhat)^{1/2}}$
  \STATE $\vphantom{\cmH_{j+1,j} \leftarrow \icft(\mHhat_{j+1,j})}$
  \STATE $\cvq_{j+1} \leftarrow \cvz \circ \cel{H}_{j+1,j}^{-1}$
	  $\vphantom{\mQhat_{j+1} \leftarrow \mZhat \mHhat_{j+1,j}^{-1}}$
  \STATE $\vphantom{\cvq_{j+1} \leftarrow \icft(\mQhat_{j+1})}$
\ENDFOR
\end{algorithmic}
\end{minipage}
\begin{minipage}[t]{0.35\linewidth}
(c) Unrolled Arnoldi for $\KK_k^{n \times n}$

 \begin{algorithmic}[1]
\REQUIRE $\cmA, \cvb, t$
\STATE $\mAhat \leftarrow \cft(\cmA)$
\STATE $\mBhat \leftarrow \cft(\cvb)$
\STATE $\mQhat_1 \leftarrow \mBhat (\mBhat^* \mBhat)^{-1/2}$
\FOR {$j = 1, \ldots, t$}
  \STATE $\mZhat \leftarrow \mAhat \mQhat_j$
  \FOR {$i = 1, \ldots, j$}
    \STATE $\mHhat_{i,j} \leftarrow \mQhat_i^* \mZhat$
			$\vphantom{\cel{H}_{i,j} \leftarrow \iprod{\cvq_i}{\cvz}}$
    \STATE $\cmH_{i,j} \leftarrow \icft(\mHhat_{i,j})$
    \STATE $\mZhat \leftarrow \mZhat - \mQhat_i \mHhat_{i,j}$
		  $\vphantom{\cvz \leftarrow \cvz - \cel{H}_{i,j} \cvq_i}$
  \ENDFOR
  \STATE $\mHhat_{j+1,j} \leftarrow (\mZhat^* \mZhat)^{1/2}$
  \STATE $\cmH_{j+1,j} \leftarrow \icft(\mHhat_{j+1,j})$
  \STATE $\mQhat_{j+1} \leftarrow \mZhat \mHhat_{j+1,j}^{-1}$
	  $\vphantom{\cvq_{j+1} \leftarrow \cvz \circ \cel{H}_{j+1,j}^{-1}}$
  \STATE $\cvq_{j+1} \leftarrow \icft(\mQhat_{j+1})$
\ENDFOR
 \end{algorithmic}

\end{minipage}

\caption{Arnoldi methods.  Algorithm (a) shows the standard Arnoldi process.
Algorithm (b) shows the Arnoldi process in the circulant algebra, and
Algorithm (c) shows the set of operations in (b) but expressed in
the Fourier space.}
\label{fig:arnoldi-circulant}
\end{figure}

This discussion raises an interesting question, why iterate
on all problems simultaneously?  One case where this is
advantageous is with sparse problems; and we return to this issue
in the concluding discussion (Section~\ref{sec:conclusion}).

\section{A Matlab package}
\label{sec:computations}


The \Matlab environment is a convenient playground
for algorithms involving matrices.  We have extended
it with a new 
class implementing the circulant algebra as a native
\Matlab object.  The
name of the resulting package and class 
is \texttt{camat}: \emph{circulant algebra matrix}.
While we will show some non-trivial examples of our package
later, let us start with a small example to give the
flavor of how it works.
\begin{quote}
\begin{verbatim}
A = cazeros(2,2,3); % creates a camat type
A(1,1) = cascalar([2,3,1]); A(1,2) = cascalar([8,-2,0]);
A(2,1) = cascalar([-2,0,2]); A(2,2) = cascalar([3,1,1]);
eig(A)          % compute eigenvalues as in Example 2;
\end{verbatim}
\end{quote}
The output, which matches the non-diagonal matrix in Example~\ref{ex:diag-evals}, is:
\begin{quote}
\begin{verbatim}
ans =
(:,:,1) =  % the first eigenvalue
    1.9401
   -1.6814
    5.7413
(:,:,2) =  % the second eigenvalue
    3.0599
    3.6814
   -1.7413
\end{verbatim}
\end{quote}


Internally, each element $\cmA \in \KK_k^{m \times n}$
is stored as a $k \times n \times m$ array \emph{along
with} its $\cft$ transformed data.  Each scalar is
stored by the $k$ parameters defining it.  To describe
this storage, let us introduce the notation
\[ \tvec(\calpha) \eqdef \sbmat{ \alpha_1 \\ \vdots \\ \alpha_k} = \tcirc(\calpha) \ve_1 , \]
to label the vector of $k$ parameters explicitly.  Thus, 
we store $\tvec(\calpha)$ for $\alpha \in \KK_k$.  This storage
corresponds to storing each scalar $\KK_k$ consecutively
in memory.   The matrix is then stored by rows.  We store
the data for the diagonal elements of the $\cft$ transformed
version in the same manner; that is, $\diag(\cft(\calpha))$
is stored as $k$ consecutive complex-valued scalars.  The
organization of matrices and vectors for the $\cft$ data
is also \emph{by row}.
The reason we store the data by row is so we can take
advantage of \Matlab's standard display operations.

At the moment, our implementation stores the elements
in both the standard and Fourier transformed space.  The rationale
behind this choice was to make it easy to investigate the results
in this manuscript.  Due to the simplicity of the operations
in the Fourier space, most of the functions on \texttt{camat}
objections use the Fourier coefficients to compute a result efficiently
and then compute
the inverse Fourier transform for the $\tvec$ representation.
Hence, rather than incurring for the Fourier transform and
inverse Fourier transform cost for each operation, we only
incur the cost of the inverse transform.  Because so few operations
are easier in the standard space, we hope to eliminate the
standard $\tvec$ storage in a future version of the code to accelerate
it even further.

We now show how the overloaded operation \texttt{eig} works
in Figure~\ref{fig:camat-eig}.  This procedure, inspired
by Theorem~\ref{thm:real-eigendecomposition}, implements
the process to get real-valued canonical eigenvalues
and eigenvectors of a real-valued matrix in
the circulant algebra. The slice \verb#Af(j,:,:)#
is the matrix $\mAhat_j^T$.  Here, the \emph{real-valued}
transpose results from
the storage-by-rows instead of the storage-by-columns.
The code proceeds by computing the eigendecomposition of each
$\mAhat_j$ with a special sort applied to produce the canonical
eigenvalues.  After all of the eigendecompositions are finished,
we need to transpose their output.  Then it feeds them to
the \texttt{ifft} function to generate the data in $\tvec$
form.

\begin{figure}
\begin{verbatim}
function [V,D] = eig(A)
% CAEIG The eigenvalue routine in the circulant algebra
Af = A.fft; k = size(Af,1);           % extract data from object
if any(imag(A.data(:))), error('specialized for real values'); end
[Vf,Df] = deal(zeros(Af));            % allocate data of size (k,n,n)
[Vf(1,:,:),Df(1,:,:)] = sortedeig(squeeze(Af(1,:,:)).');
for j=2:floor(k/2)+1
  [Vf(j,:,:),Df(j,:,:)] = sortedeig(sqeeze(Af(j,:,:)).');
  if j~=k/2+1                         % skip last when k is even
    Vf(k-j+2,:,:) = conj(Vf(j,:,:)); Df(k-j+2,:,:) = conj(Df(j,:,:));
  end
end
% transpose all the data back.
for j=1:k, Vf(j,:,:) = Vf(j,:,:).'; Df(j,:,:) = Df(j,:,:).'; end
V = camatcft(ifft(Vf),Vf);            % create classed output
D = camatcft(ifft(Df),Df);
function [V,D]=sortedeig(A)
[V,D] = eig(A); d = diag(D); [ignore p] = sort(-abs(d));
V = V(:,p); D = D(p,p);               % apply the sort
\end{verbatim}
\caption{The implementation of the eigenvalue computation
in our package.  Please see the discussion in the text.}
\label{fig:camat-eig}
\end{figure}

In a similar manner, we overloaded the standard assignment
and indexing operations e.g.\ \texttt{a = A(i,j); A(1,1) = a};
the standard Matlab arithmetic operations \texttt{+, -, *, /,}
\verb#\#;
and the functions \texttt{abs, angle, norm, conj, diag, eig,
hess, mag, norm, numel, qr, rank, size, sqrt, svd}.

All of these operations have been mentioned or are self explanatory, 
except \texttt{mag}. It  is a {\em magnitude
function}, and we discuss it in detail in~\ref{sec:magnitude}.

Using these overloaded operations, implementing the power method
is straightforward; see Figure~\ref{fig:camat-power}. 
We note that the power method and Arnoldi methods can be further optimized 
by implementing them directly in Fourier space. This remains as an item for future work.


\begin{figure}
\begin{verbatim}
for iter=1:maxiter
  Ax = A*x;
  lambda = x'*Ax;
  x2 = (1./ norm(Ax))*Ax;
  delta = mag(norm(1./angle(x(1))*x-1./angle(x2(1))*x2));
  if delta<tol, break, end
end
\end{verbatim}
\caption{The implementation of the power method using our package.}
\label{fig:camat-power}
\end{figure}



\section{Numerical examples}
\label{sec:example}

In this section, we present a numerical example using
the code we described in Section~\ref{sec:computations}.  The
problem we consider is the Poisson equation on a regular
grid with a mixture of periodic and fixed boundary conditions:
\[ -\Delta u(x,y) = f(x,y) \qquad u(x,0) = u(x,1), u(0,y) = y(1,y) = 0 \qquad (x,y) \in [0,1] \times [0,1]. \]
Consider a uniform mesh and the standard 5-point discrete Laplacian:
\[ -\Delta u(x_i,y_j) \approx -u(x_{i-1},y_j) -u(x_i,y_{j-1}) + 4 u(x_i,y_j) - u(x_{i+1},y_j) - u(x_i,y_{j+1}). \]
After applying the boundary conditions and organizing the
unknowns of $u$
in $y$-major order, an approximate solution $u$ is given
by solving an $N(N-1)\times N(N-1)$ block-tridiagonal, circulant-block system:
\[ \underbrace{\bmat{
    \mC & -\mI \\
    -\mI & \mC & \ddots \\
    & \ddots & \ddots & -\mI \\
    & & -\mI & \mC \\
   }}_{\mA} \underbrace{\bmat {
    \vu(x_1, \cdot) \\
    \vu(x_2, \cdot) \\
    \vdots \\
    \vu(x_{N-1}, \cdot)
  }}_{\vu} = \underbrace{\bmat{
    \vf(x_1, \cdot) \\
    \vf(x_2, \cdot) \\
    \vdots \\
    \vf(x_{N-1}, \cdot)
  }}_{\vf}, \qquad
  \mC = \underbrace{\bmat{
    4 & -1 & & -1\\
    -1 & 4 & \ddots \\
    & \ddots & \ddots & -1 \\
    -1 & & -1 & 4 \\
   }}_{N \times N} ,
  \]
that is, $\mA \vu = \vf$.
Because of the circulant-block structure, this system
is equivalent to
\[ \cmA \circ \cvu = \cvf \]
where $\cmA$ is an ${N-1} \times {N-1}$ matrix of $\KK_{N}$ elements,
$\cvu$ and $\cvf$ have compatible sizes, and
\[ \mA =\tcirc(\cmA) \qquad  \vu =  \tvec(\cvu)  \qquad \vf = \tvec(\cvf). \]  We now investigate this matrix and linear system
with $N = 50$.

\subsection{The power method}

We first study the behavior of the power method
on $\cmA$.
The canonical eigenvalues of $\cmA$ are
\[ \clambda_j = \csbmat{ 4 + 2\cos(j \pi / {N}), -1, 0, \ldots, 0, -1 }. \]
To see this result, let
$\clambda(\mu) =\csbmat{ \mu, -1, 0, \ldots, 0, -1 }.$
Then
\[ (\cmA - \clambda(\mu) \circ \cmI)
 = \bmat{ (4-\mu) \circ \cel{1} & -1 \circ \cel{1} \\
          -1 \circ \cel{1} & (4-\mu) \circ \cel{1} & \ddots \\
          & \ddots & \ddots & -1 \circ \cel{1} \\
          & & -1 \circ \cel{1} & (4-\mu) \circ \cel{1} }.
\]
The \emph{canonical} eigenvalues of $\cmA - \clambda(\mu) \circ \cmI$
can be determined by choosing $\mu$ to be an eigenvalue of $\mT =
\texttt{tridiag}(-1,4,-1)$.  These are given
 by setting $\mu = 4 + 2\cos(j \pi / {N})$,
where each choice $j=1, \ldots, N-1$ produces a canonical
eigenvalue $\clambda_j$. 
From these canonical eigenvalues, we can estimate the convergence
behavior of the power method.  Recall that the algorithm runs
independent power methods in Fourier space.  Consequently, these
rates are given by $\hat{\lambda}_2/\hat{\lambda}_1$ for each
matrix $\mAhat_j$.  To state these ratios compactly,
let $\gamma_1 = 4 + 2\cos( \pi / {N})$ and $\gamma_2 =
4 + 2\cos(2 \pi / {N})$; also let $\delta_j = 2 \cos (-\pi + 2\pi(j-1)/N)$.  
For $N$ even, 
\[
\begin{aligned}
\cft(\clambda_1) & =
  \diag\sbmat{ \gamma_1 + \delta_1, & \ldots, \gamma_1 + \delta_N} \\
\cft(\clambda_2) & =
  \diag\sbmat{ \gamma_2 + \delta_1, & \ldots, \gamma_2 + \delta_N}	
\end{aligned}
\]	
Thus, the convergence ratio for $\mAhat_j$ is $(\gamma_2 + \delta_j)/(\gamma_1 + \delta_j)$.
The largest ratio (fastest converging)  corresponds to the smallest value of $\delta_j$, which is $\delta_1$.  The smallest ratio (slowest
converging) corresponds to the largest value of $\delta_j$,
which is $\delta_{N/2+1}$ in this case.  
(This choice will slightly change in an obvious manner if $N$ is odd.)  Evaluating these ratios yields
\[ \begin{aligned}
\min_j \frac{\lambda_2(\mAhat_j)}{\lambda_1(\mAhat_j)} & = \frac{\gamma_2 + \delta_1}{\gamma_1 + \delta_1} = \frac{2 + 2\cos(2\pi/N)}{2 + 2\cos(\pi/N)} && \qquad \text{(fastest)} \\
\max_j \frac{\lambda_2(\mAhat_j)}{\lambda_1(\mAhat_j)} & = \frac{\gamma_2 + \delta_{N/2+1}}{\gamma_1 + \delta_{N/2+1}} = \frac{6 + 2\cos(2\pi/N)}{6 + 2\cos(\pi/N)} && \qquad \text{(slowest)}. \\
\end{aligned}
\]
Based on this analysis, we expect the eigenvector to
converge linearly with the rate $\frac{6 + 2\cos(2\pi/N)}{6 + 2\cos(\pi/N)}$.
By the standard theory for the power method, expect the eigenvalues to converge twice as fast.

Let $\cel{\rho}$ be the eigenvector
change measure from equation~\eqref{eq:power-convergence}.
In Figure~\ref{fig:experiment-power}, we first show how
the maximum absolute value of the Fourier coefficients
in $\cel{\rho}$ behaves (the red line).
Formally, this measure is $\normof[1]{\cft(\cel{\rho})}$, i.e.,\ the
maximum element in the diagonal matrix.  We also show 
how each Fourier component of the eigenvalue
converges to the Fourier components of $\clambda_1$ (each gray line).
That is, let $\cel{\mu}\itn{i}$
be the Rayleigh quotient ${\cvx\itn{i}}^* \circ \cmA \circ \cvx\itn{i}$
at the $i$th iteration.  Then these lines are the $N$ values of
$\diag(\cft(\tabs(\cel{\mu}\itn{i} - \clambda_1)))$.
The results validate the
theoretical predictions, and the eigenvalue does
indeed converge to $\clambda_1$.


\subsection{The Arnoldi method}

We next investigate computing $\cvu$ using the Arnoldi
method applied to $\cmA$.  In this case, $f(x,y)$ to
be $1$ at $x_{25},y_2$ and $0$ elsewhere.  This corresponds
to a single non-zero in $\tvec(\cvf)$ with value $1/N^2$.
With this right-hand side,
the procedure we use is identical
to an unoptimized GMRES procedure.  Given a $t$-step Arnoldi
factorization starting from $\cvf$, we estimate
\[ \cvu\itn{t} \approx
 \cmQ_t \circ \mathop{\mathrm{arg\,min}}_{\cvy \in\KK_k} \normof{\cmH_{t+1,t} \circ \cvy - \cbeta \circ \cve_1},
\]
where $\cbeta=\normof{\cvf}$.
We solve the least-squares problem by solving each problem
independently in the Fourier space -- as has become standard throughout
this paper.  Let $\cel{\rho} = \normof{\cvf - \cmA \circ \cvu\itn{t}}$.
Figure~\ref{fig:experiment-arnoldi} shows (in red) the
magnitude of the residual as a function of the Arnoldi factorization
length $t$, which is $\normof[1]{\cft(\cel{\rho})}$.  
The figure also shows (in gray) the magnitude of the error in the
$j$th Fourier coefficient; these lines are the $N$ values of
$\diag(\cft(\normof{\cvu - \cvu\itn{t}}))$.
In Fourier space,
these values measure the error in  each individual Arnoldi
process.

What the figure shows is that the residual suddenly converges
at the $26$th iteration.  
This is in fact theoretically expected \cite{saad2003-book}, because
each matrix $\mAhat_j$ has $N/2+1=26$ distinct eigenvalues.
In terms of measure the individual errors (the gray lines), 
some converge
rapidly, and some do not seem to converge at all until the
Arnoldi process completes at iteration 26.  This exemplifies
how the overall behavior is governed by the worst behavior
in any of the independent Arnoldi processes.

\begin{figure}[t]
\begin{minipage}{0.46\linewidth}
\includegraphics[width=\linewidth]{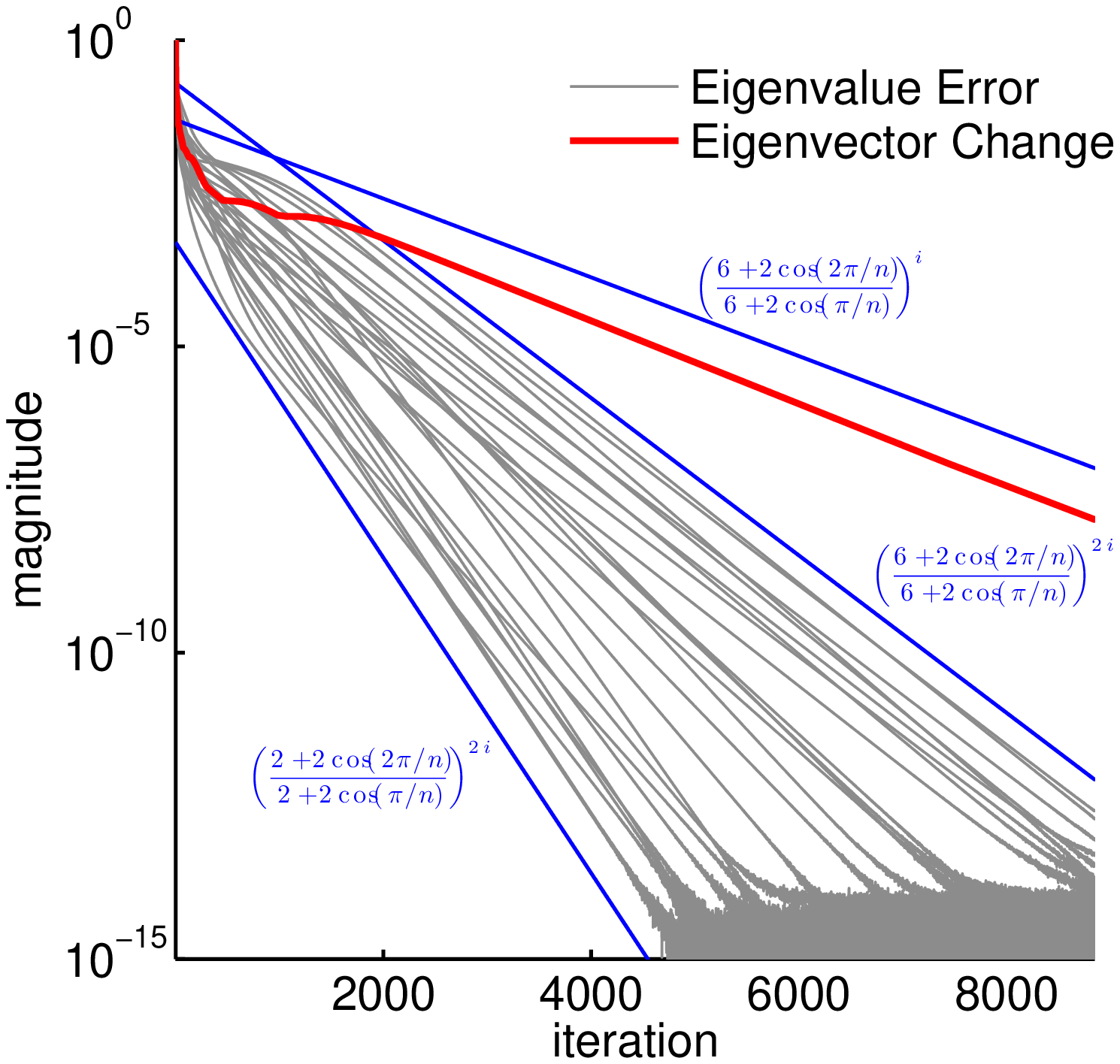}
\caption{The convergence behavior of the power
method in the circulant algebra. The gray lines show
the error in the each eigenvalue in Fourier space.
These curves track the predictions made based
on the eigenvalues as discussed in the text.
The red line shows the magnitude of the change
in the eigenvector.  We use this as the stopping
criteria.  It also decays as predicted by the
ratio of eigenvalues. The blue fit lines have
been visually adjusted to match the behavior
in the convergence tail.   }
\label{fig:experiment-power}
\end{minipage}
\hfil
\begin{minipage}{0.46\linewidth}
\includegraphics[width=\linewidth]{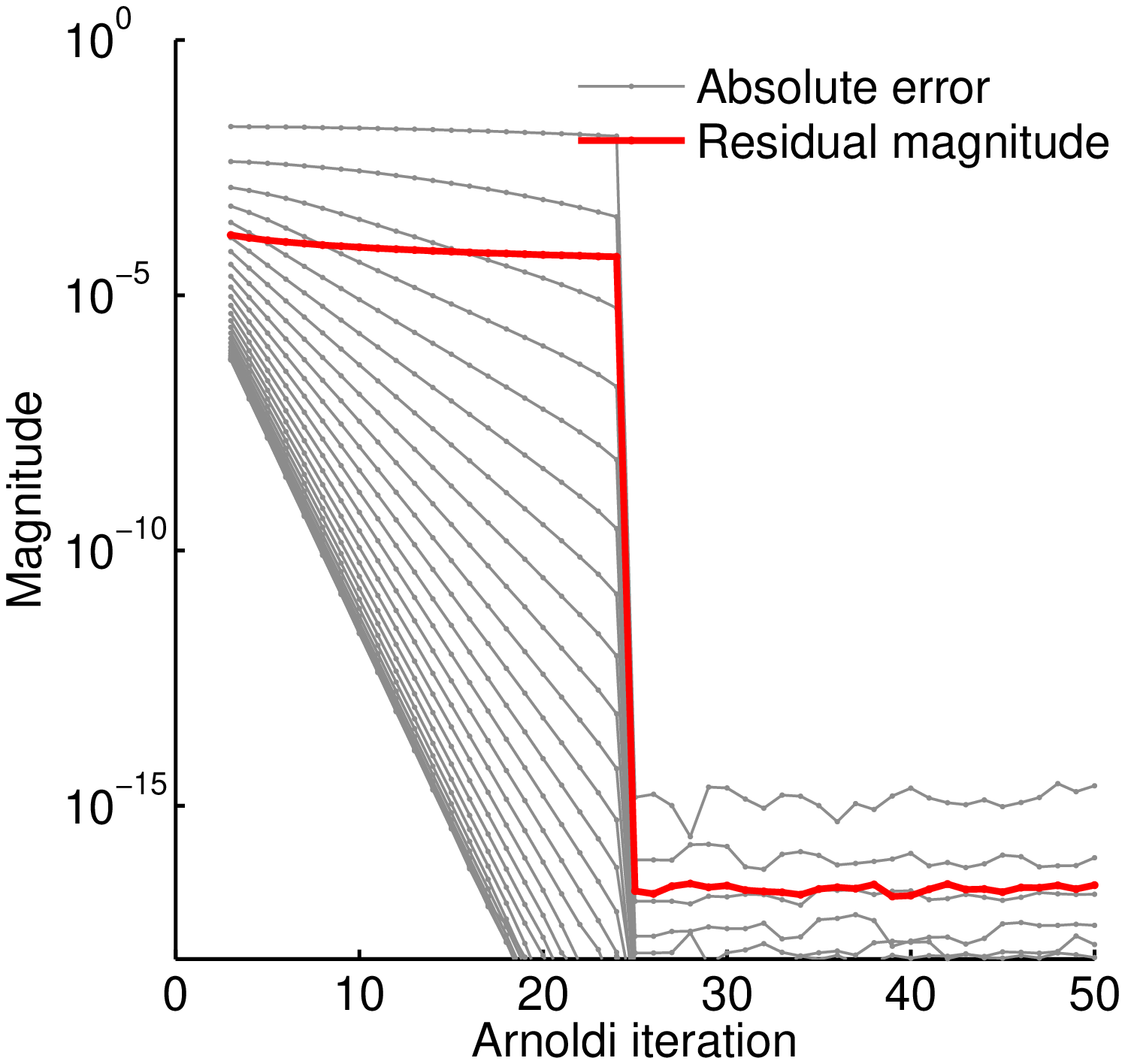}
\caption{The convergence behavior of a
GMRES procedure using the circulant Arnoldi process.
The gray lines show the error in each Fourier
component and the red line shows the magnitude
of the residual.  We observe poor convergence
in one Fourier component; until the Arnoldi basis
captures all of the eigenvalues after $N/2+1=26$ iterations.
These results show how the two computations are
performing individual power methods or Arnoldi processes
in Fourier space.}
\label{fig:experiment-arnoldi}
\end{minipage}

\end{figure}

\section{Summary}
\label{sec:conclusion}

%

We have extended the circulant algebra, introduced
by \citet{kilmer2008-circ-tensor-svd}, with new
operations 
to pave the way for
iterative algorithms, such as the power method
and the Arnoldi iteration that we introduced.
These operations
provided key tools to build a \Matlab package 
to investigate these iterative algorithms for this paper.
Furthermore, we used the fast Fourier transform to accelerate 
these operations, 
and as a key analysis tool for eigenvalues
and eigenvectors. 
In the Fourier space the operations and algorithms decouple into individual problems.
We observed this for the inner product, eigenvalues,
eigenvectors, the power method, and the Arnoldi iteration.
We also found that this decoupling explained the behavior 
in a numerical example. 

Given that decoupling is such a powerful computational
and analytical tool, a natural question that arises is when 
it is useful to employ the original circulant formalism, rather than work in the Fourier space.
For dense computations, it is likely that working entirely
in Fourier space is a superior approach.  However, for sparse computations,
such as the system $\cmA \circ \cvu = \cvf$ explored in
Section~\ref{sec:example}, such a conclusion is unwarranted.
That example is sparse both in the matrix over circulants,
and in the individual circulant arrays.  When thought of
as a cube of data, it is sparse in any way of slicing it
into a matrix. After this matrix $\cmA$ is transformed
to the Fourier space, it loses its sparsity in the
third-dimension; each sparse scalar $\cel{A}_{i,j}$
becomes a dense array. In this case, retaining the coupled
nature of the operations and even avoiding most of the
Fourier domain may allow better scalability in terms
of total memory usage.

An interesting topic for future work is exploring other
rings besides the ring of circulants.  One obvious
candidate is the ring of symmetric circulant matrices.
In this ring, the Fourier coefficients are always
real-valued.  Using this ring avoids the algebraic and computational complexity
associated with complex values in the Fourier transforms.

We have made all of code and experiments available
to use and reproduce our results:
\url{http://stanford.edu/~dgleich/publications/2011/codes/camat}.

%
%
%
%
%
%
%

\section*{References}
\bibliographystyle{elsarticle-harv}
\bibliography{all-bibliography}

\appendix

\section{The circulant scalar magnitude}
\label{sec:magnitude}

This section describes another operation
we extended to the circulant algebra.  Eventually,
we replaced it with our ordering (Definition~\ref{def:ordering}),
which is more powerful as we justify below.
However, it plays
a role in our \Matlab package, and thus we describe
the rationale for our choice of magnitude function here.

For scalars in $\RR$, the
magnitude is often called the absolute value.
Let $\alpha, \beta \in \RR$.  The absolute value
has the the property $\absof{\alpha \beta} = \absof{\alpha} \absof{\beta}$.
We have already introduced an absolute value function,
however.  Here, we wish to define a notion of magnitude that
produces a scalar in $\RR$ to indicate the size of an
element.  Such a function will have norm-like flavor because
it must represent the aggregate magnitude of $k$ values with
a single real-valued number.  Thus, finding a function to satisfy
$\absof{\calpha \circ \cbeta} = \absof{\calpha} \absof{\cbeta}$
exactly is not possible.  Instead, we seek a
function $g : \KK_k \mapsto \RR $ such that
\begin{enumerate}
 \item $g(\calpha) = 0$ if and only if $\calpha = 0$,
 \item $g(\calpha \circ \cbeta) \le g(\calpha) g(\cbeta)$,
 \item $g(\calpha + \cbeta) \le g(\calpha) + g(\cbeta)$.
\end{enumerate}
The following result shows that there is a large class of
such magnitude functions.

\begin{result}
 Any sub-multiplicative matrix norm $\normof{\mA}$ defines a magnitude
 function $ g(\calpha) = \normof{\tcirc(\calpha)}. $
\end{result}

This result follows because the properties of the function
$g$ are identical to the requirements of a sub-multiplicative
matrix norm applied to $\tcirc(\calpha)$.
Any matrix norm induced by a vector norm is sub-multiplicative.
In particular, the matrix $1$, $2$, and $\infty$ norms are
all sub-multiplicative.
Note that for circulant matrices
both the matrix $1$ and $\infty$ norms are equal to the 1-norm
of any row or column, i.e., $\normof[1]{\tvec(\alpha)}$ is a valid
magnitude.  Surprisingly, the 2-norm of the vector of parameters,
that is $\normof[2]{\tvec(\calpha)}$, is not.  For a counterexample,
let $\calpha = \csbmat{ 1 & 2 }, \cbeta = \csbmat{ 2 & 4 }$.  Then
$\calpha \circ \cbeta = \csbmat{ 8 & 10 }$ and
$\normof[2]{\tvec(\calpha \circ \cbeta)} = \sqrt{164} >
\normof[2]{\tvec(\calpha)} \normof[2]{\tvec(\cbeta)} = \sqrt{100}.$
For many practical computations, we use the matrix $2$-norm
of $\tcirc(\calpha)$ as the magnitude function.  Thus,
\[ \absof{\calpha} \eqdef \normof[2]{\tcirc(\calpha)} = \normof[1]{\cft(\calpha)}. \]
This choice has the following relationship with our ordering:
\[ \tabs(\calpha) \le \tabs(\cbeta) \qquad  \Rightarrow \qquad \absof{\calpha} \le \absof{\cbeta}. \]
We implement this operation as the \texttt{mag} function
in our \Matlab package.

\end{document}

%% file: ggv2010-setup.tex
\usepackage{ragged2e}
\usepackage{xspace}
\usepackage{amssymb}
\usepackage{amsmath}
\usepackage{amsthm}
\usepackage{url}
\usepackage{tabularx}
\usepackage{booktabs}
\usepackage{paralist}
\usepackage{xcolor}
\graphicspath{{./}{./figures/}}

\renewcommand{\cite}{\citep}

\usepackage{ushort}
\usepackage{mathdots}

\providecolor{TufteRed}{rgb}{0.8,0,0}
\providecolor{tuftered}{rgb}{0.8,0,0}
\providecolor{subtleblue}{rgb}{0.8,0,0}

\newtheorem{theorem}{Theorem}

\newtheorem{proposition}[theorem]{Proposition}
\newdefinition{remark}{Remark}
\newdefinition{result}{Result}
\newdefinition{definition}[theorem]{Definition}
\newtheorem{example}[theorem]{Example}

\DeclareMathOperator{\tcirc}{\texttt{circ}}
\DeclareMathOperator{\tspan}{\texttt{span}}
\DeclareMathOperator{\tabs}{\texttt{abs}}
\DeclareMathOperator{\tvec}{\texttt{vec}}
\DeclareMathOperator{\rank}{rank}%
\DeclareMathOperator{\diag}{diag}%
\providecommand{\ii}{\imath}
\providecommand{\kron}{\otimes}
\providecommand{\eqdef}{\equiv}

\providecommand{\RR}{\mathbb{R}}
\providecommand{\KK}{\mathbb{K}}
\providecommand{\CC}{\mathbb{C}}

\providecommand{\normof}[2][]{\left\| #2 \right\|_{#1}}%

\providecommand{\itn}[1]{^{(#1)}}%

\providecommand{\absof}[1]{\left| #1 \right|}
\newcommand{\niprod}[2]{\langle {#1}, {#2} \rangle}

\newcommand{\iprod}{\niprod}

\newcommand{\conj}[1]{\overline{#1}}

\providecommand{\bmat}[1]{\begin{bmatrix} #1 \end{bmatrix}}

\providecommand{\sbmat}[1]{\left[\begin{smallmatrix} #1 \end{smallmatrix}\right]}

\DeclareMathOperator{\sign}{sign}

\providecommand{\mat}{\boldsymbol}
\renewcommand{\vec}{\mathbf}

\providecommand{\eye}{\mat{I}}
\providecommand{\mA}{\ensuremath{\mat{A}}}

\providecommand{\mC}{\ensuremath{\mat{C}}}
\providecommand{\mD}{\ensuremath{\mat{D}}}

\providecommand{\mF}{\ensuremath{\mat{F}}}
\providecommand{\mG}{\ensuremath{\mat{G}}}
\providecommand{\mH}{\ensuremath{\mat{H}}}
\providecommand{\mI}{\ensuremath{\mat{I}}}

\providecommand{\mP}{\ensuremath{\mat{P}}}
\providecommand{\mQ}{\ensuremath{\mat{Q}}}
\providecommand{\mR}{\ensuremath{\mat{R}}}

\providecommand{\mT}{\ensuremath{\mat{T}}}

\providecommand{\mX}{\ensuremath{\mat{X}}}
\providecommand{\mY}{\ensuremath{\mat{Y}}}

\providecommand{\mhat}[1]{\ensuremath{\mat{\hat{#1}}}}
\providecommand{\vhat}[1]{\ensuremath{\vec{\hat{#1}}}}
\providecommand{\mAhat}{\mhat{A}}
\providecommand{\vxhat}{\vhat{x}}

\providecommand{\vb}{\ensuremath{\vec{b}}}

\providecommand{\ve}{\ensuremath{\vec{e}}}
\providecommand{\vf}{\ensuremath{\vec{f}}}

\providecommand{\vq}{\ensuremath{\vec{q}}}

\providecommand{\vu}{\ensuremath{\vec{u}}}
\providecommand{\vv}{\ensuremath{\vec{v}}}

\providecommand{\vx}{\ensuremath{\vec{x}}}
\providecommand{\vy}{\ensuremath{\vec{y}}}
\providecommand{\vz}{\ensuremath{\vec{z}}}

%% file: ggv2010-custom.tex
\newcommand{\fftm}{\mF}

\DeclareMathOperator{\fft}{\texttt{cft}}
\DeclareMathOperator{\ifft}{\texttt{icft}}
\DeclareMathOperator{\icft}{\texttt{icft}}
\DeclareMathOperator{\cft}{\texttt{cft}}

\newcommand{\ceilof}[1]{\lceil #1 \rceil}

\renewcommand{\circeq}{\leftrightarrow}
\newcommand{\cel}[1]{\ushort{#1}}

\newcommand{\csbmat}[1]{\left\{\! \begin{smallmatrix} #1
\end{smallmatrix} \! \right\}}
\newcommand{\smallmath}[1]{\begin{smallmatrix}#1\end{smallmatrix}}
\newcommand{\celm}[1]{\cel{\mat{#1}}}
\newcommand{\celv}[1]{\cel{\vec{#1}}}

\newcommand{\calpha}{\ensuremath{\cel{\alpha}}}
\newcommand{\cbeta}{\ensuremath{\cel{\beta}}}
\newcommand{\clambda}{\ensuremath{\cel{\lambda}}}

\newcommand{\cone}{\ensuremath{\cel{1}}}

\newcommand{\cx}{\ensuremath{\cel{x}}}

\newcommand{\cA}{\ensuremath{\cel{A}}}

\newcommand{\cI}{\ensuremath{\cel{I}}}

\newcommand{\cva}{\ensuremath{\celv{a}}}
\newcommand{\cvb}{\ensuremath{\celv{b}}}

\newcommand{\cve}{\ensuremath{\celv{e}}}
\newcommand{\cvf}{\ensuremath{\celv{f}}}

\newcommand{\cvq}{\ensuremath{\celv{q}}}

\newcommand{\cvu}{\ensuremath{\celv{u}}}

\newcommand{\cvx}{\ensuremath{\celv{x}}}
\newcommand{\cvy}{\ensuremath{\celv{y}}}
\newcommand{\cvz}{\ensuremath{\celv{z}}}

\providecommand{\cmA}{\ensuremath{\celm{A}}}

\providecommand{\cmH}{\ensuremath{\celm{H}}}
\providecommand{\cmI}{\ensuremath{\celm{I}}}

\providecommand{\cmQ}{\ensuremath{\celm{Q}}}

\providecommand{\cmX}{\ensuremath{\celm{X}}}